\pdfoutput=1 

\documentclass[a4paper,oneside,leqno,10pt,final]{article}

\usepackage{a4wide}
\usepackage[utf8]{inputenc}
\usepackage[a4paper]{geometry}

\usepackage{titlesec}

\usepackage{bm}

\usepackage{enumitem}
\setlist{itemsep=3pt,parsep=1pt,topsep=3pt,partopsep=1pt}
\setlist[enumerate,1]{label=(\arabic*)}
\setlist[enumerate,2]{label=(\alph*)}

\usepackage[dvipsnames]{xcolor}

\usepackage{graphicx}

\usepackage{mathtools}
\mathtoolsset{showonlyrefs}

\usepackage[only,llbracket,rrbracket]{stmaryrd}

\usepackage[T1]{fontenc}

\usepackage[p,osf]{scholax}
\usepackage{amsmath,amsthm}
\usepackage[scaled=1.075,ncf,vvarbb]{newtxmath}

\usepackage{chngcntr}

\usepackage{url}
\usepackage[colorlinks=true,linkcolor=blue,citecolor=Violet,linktocpage=false]{hyperref}

\usepackage{authoraftertitle}

\swapnumbers

\theoremstyle{plain}
\newtheorem{theorem}{Theorem}[section]
\newtheorem{lemma}[theorem]{Lemma}
\newtheorem{corollary}[theorem]{Corollary}

\newtheorem*{theorem*}{Theorem}
\newtheorem*{corollary*}{Corollary}

\theoremstyle{definition}
\newtheorem{definition}[theorem]{Definition}

\newtheorem{miniremark}[theorem]{}

\newtheorem*{notation*}{Notation}

\theoremstyle{remark}

\swapnumbers

\newtheoremstyle{newremark}  
{}                           
{}                           
{}                           
{}                           
{}                           
{.}                          
{ }                          
{{\bfseries \thmnumber{#2}}{\itshape \thmname{ #1}}\thmnote{ (#3)}} 

\theoremstyle{newremark}
\newtheorem{remark}[theorem]{Remark}

\newcommand{\End}[1]{ \mathrm{End}({#1}) }

\newcommand{\ccspace}[1]{\mathscr{K}(#1)}

\newcommand{\transpose}[1]{{#1}^{*}}

\newcommand{\project}[1]{{#1}_\natural}

\newcommand{\eqproject}[1]{({#1})_\natural}

\newcommand{\perpproject}[1]{{#1}_{\natural}^{\perp}}

\newcommand{\pgrass}[2]{\mathbf{G}_{\natural}(#1,#2)}
\newcommand{\grass}[2]{\mathbf{G}(#1,#2)}

\newcommand{\measureball}[2]{{#1}\,{#2}}

\newcommand{\Var}[1]{\mathbf{V}_{#1}}     
\newcommand{\RVar}[1]{\mathbf{RV}_{#1}}   
\newcommand{\IVar}[1]{\mathbf{IV}_{#1}}   
\DeclareMathOperator{\VarTan}{VarTan} 
     
\newcommand{\var}[1]{\mathbf{v}_{#1}}     

\newcommand{\adim}{n}
\newcommand{\vdim}{k}
\newcommand{\codim}{n-k}

\newcommand{\oball}[2]{\mathbf{U}(#1,#2)}

\newcommand{\cball}[2]{\mathbf{B}(#1,#2)}

\newcommand{\sphere}[1]{\mathbb{S}^{#1}}

\ifdefined\textint
    \renewcommand{\textint}[2]{{\textstyle\int_{#1}^{#2}}}
\else    
    \newcommand{\textint}[2]{{\textstyle\int_{#1}^{#2}}}
\fi

\ifdefined\textfint
    \renewcommand{\textfint}[2]{{\textstyle\fint_{#1}^{#2}}}
\else    
    \newcommand{\textfint}[2]{{\textstyle\fint_{#1}^{#2}}}
\fi

\ifdefined\textsum
    \renewcommand{\textsum}[2]{{\textstyle\sum_{#1}^{#2}}}
\else  
    \newcommand{\textsum}[2]{{\textstyle\sum_{#1}^{#2}}}
\fi

\ifdefined\textprod
  \renewcommand{\textprod}[2]{{\textstyle\prod_{#1}^{#2}}}
\else  
  \newcommand{\textprod}[2]{{\textstyle\prod_{#1}^{#2}}}
\fi

\newcommand{\natp}{\mathscr{P}}
\newcommand{\nat}{\natp \cup \{0\}}

\newcommand{\integers}{\mathbf{Z}}

\newcommand{\R}{\mathbf{R}}

\newcommand{\CF}[1]{ \raisebox{\depth}{$\boldsymbol{\chi}$}_{#1} }

\newcommand{\Lploc}[1]{\mathbf{L}_{#1}^{\mathrm{loc}}}

\DeclareMathOperator{\dist}{dist}

\newcommand{\mcv}[1]{ \mathbf{h}_{{#1}} }

\newcommand{\wmcv}[1]{ \overline{\mathbf{h}}_{{#1}} }

\newcommand{\HM}{\mathscr{H}}

\newcommand{\density}{\boldsymbol{\Theta}}

\newcommand{\unitmeasure}[1]{\boldsymbol{\alpha}(#1)}

\newcommand{\restrict}{ \mathop{ \rule[1pt]{.5pt}{6pt} \rule[1pt]{4pt}{0.5pt} }\nolimits }

\newcommand{\ud}{\ensuremath{\,\mathrm{d}}}

\newcommand{\uD}{\ensuremath{\mathrm{D}}}
\newcommand{\Der}{\uD}

\DeclareRobustCommand{\rchi}{{\mathpalette\irchi\relax}}
\newcommand{\irchi}[2]{\raisebox{\depth}{$#1\chi$}}

\newcommand{\id}[1]{\bm{1}_{#1}}

\newcommand{\lIm}{ [ }
\newcommand{\rIm}{ ] }

\newcommand{\scale}[1]{\boldsymbol{\mu}_{#1}}
\newcommand{\trans}[1]{\boldsymbol{\tau}_{#1}}

\newcommand{\tbwedge}{{\textstyle \boldsymbol{\bigwedge}}}

\newcommand{\tbcup}{{{\textstyle \bigcup}}}

\newcommand{\Clos}[1]{\mathop{\mathrm{Clos}}#1}

\newcommand{\VF}{\mathscr{X}}

\DeclareMathOperator{\trace}{trace}

\DeclareMathOperator{\Hom}{Hom}

\newcommand{\Bdry}{\partial}

\DeclareMathOperator{\lin}{span}

\newcommand{\cnt}[1]{\mathscr{C}^{#1}}

\newcommand{\orthproj}[2]{\mathbf{O}^\ast({#1},{#2})}

\DeclareMathOperator{\Int}{Int}

\newcommand{\grad}{\nabla}

\DeclareMathOperator{\spt}{spt}

\DeclareMathOperator{\Tan}{Tan}

\DeclareMathOperator{\Nor}{Nor}

\DeclareMathOperator{\Lip}{Lip}

\DeclareMathOperator{\sgn}{sgn}

\newcommand{\without}{\mathbin{\smallsetminus}}

\DeclareMathOperator{\im}{im}

\renewcommand{\adim}{{n+1}}
\renewcommand{\vdim}{n}
\renewcommand{\codim}{1}

\makeatletter
\newcommand*{\nsubsection}[1]{%
  \subsection*{#1}%
  \addcontentsline{toc}{subsection}{#1}
  \NR@gettitle{#1}%
}
\makeatother

\makeatletter
\let\c@equation\c@enumi

\makeatother

\counterwithin*{equation}{theorem}
\counterwithin*{enumi}{theorem}

\title{Perpendicularity and Locality \\for Codimension-One 
  Varifolds\\ with Bounded Anisotropic Mean Curvature}

\author{Sławomir Kolasiński \and Mario Santilli}

\hypersetup{
  unicode=true,
  pdfauthor={\MyAuthor},
  pdftitle={\MyTitle},
  pdfsubject={},
  pdfkeywords={},
  pdfproducer={},
  pdfcreator={}
  pdfinfo={
    orcid_SK={0000-0002-7122-7275},
    orcid_MS={0000-0002-9774-9776},
  }
}

\begin{document}

\maketitle
\tableofcontents

\begin{abstract}
 Suppose $ F $ is an integrand associated with a uniformly convex $ \cnt{3} $-norm, and $ V $ is a $ n $-dimensional varifold in an open subset of $ \R^{n+1} $ such that  $ \HM^n \restrict \spt \| V \| $ is absolutely
continuous with respect to $ \| V \| $ and  the  mean $ F $-curvature  $ \mcv{F}(V, \cdot) $ is bounded in $ L^\infty $. In our previous result \cite{KolSan2} we prove that $ \spt \| V \| $ is $ \cnt{2} $-rectifiable and the $ \cnt{1} $-regular part $ M $ of $ \spt \| V \| $ coincides $ \HM^n $ almost everywhere with the unit-density stratum of $ V $. In this paper we prove that $ \mcv{F}(V,a) \in \Nor(M,a) $ for $ \HM^n $ a.e.\ $ a \in M $ and that $ \mcv{F}(V, \cdot) $ agrees with the approximate  mean $ F $-curvature coming from the $ \cnt{2} $-rectifiable covering of $ M $. These results provide anisotropic extensions of well known theorems in the Euclidean setting by Brakke, Sch\"atzle and Ambrosio-Masnou. 
\end{abstract}

\paragraph*{\small MSC-classes 2020:}{53A35, 49Q20}

\section{Introduction}
\label{sec:intro}

In our previous paper~\cite{KolSan2} we prove quadratic flatness and $ \cnt{2} $-rectifiability for
the support of $ n $-dimensional varifolds on open subsets of $ \R^{\adim} $ with bounded
anisotropic mean curvature, under the additional hypothesis that the support has locally finite
$\HM^{\vdim}$~measure. In case the varifold $ V $ is integral and satisfies the \emph{absolute
  continuity hypothesis}, i.e., $ \HM^{\vdim} \restrict \spt \| V \| \ll \| V \| $, this allows to
conclude, in combination with Allard's anisotropic regularity theorem \cite{Allard1986}, that the
unit density layer of~$V$ is a $ \cnt{1, \alpha} $ manifold; cf.~\cite[Theorem~1.2]{KolSan2}. The
present paper explores a fundamental consequence of this result, by establishing perpendicularity of
the anisotropic mean curvature vector on this unit-density layer, and compatibility with the
approximate second order structures (i.e.\ locality of the mean curvature vector).

\begin{theorem}
    \label{main:locality}
    Let $ \phi $ be a uniformly convex $ \cnt{3} $-norm on $ \R^{\adim} $ and let
    $ F : \grass{n+1}{n} \rightarrow \mathbf{R} $ be the elliptic integrand naturally associated
    with $ \phi $ (cf.~\ref{mr:anisotropic_mean_curvature_vector}). Suppose
    $\Omega \subseteq \R^{\adim} $ is open, $ 0 \leq H < \infty $ and $V $ is a
    $ \vdim $-dimensional integral varifold in $ \Omega $ such that
    \begin{gather}
        \text{$ \| \delta_F V \| \leq H \| V \| $}\\
        	\text{and $ \HM^{\vdim} \restrict \spt \| V \| $ is absolutely continuous with respect to $ \| V \| $} \,.
         \end{gather}
     Define $ Q = \spt \| V \| \cap \bigl\{ a  : \density^{\vdim}(\| V \|, a) = 1 \bigr\} $.
    
  Then  there holds
    \begin{displaymath}
        \mcv{F}(M,a) = \mcv{F}(V,a) \in \Nor(M,a)
        \quad \text{for $\HM^{\vdim}$~almost all $a \in Q \cap M$}\,
    \end{displaymath}
whenever $ M \subseteq \Omega$ is an embedded $\cnt{2}$-hypersurface with $ \HM^{\vdim}(M) < \infty $.
\end{theorem}

If $ M \subseteq \Omega $ is a $ n $-dimensional $ \cnt{2} $-hypersurface, then an explicit formula
for the mean $ F $-curvature $ \mcv{F}(M, \cdot) $ of $ M $ (which in particular implies that
$ \mcv{F}(M, \cdot) $ is normal to $ M $) has been computed in \cite[Proposition
2.1]{DePhilippis2019} employing a differential-geometric approach that relies on the
$ \cnt{2} $-regularity. Under the hypothesis of Theorem \ref{main:locality}, the set $ Q $ is proved
to be only $ \cnt{1, \alpha} $-regular almost everywhere, hence no differential-geometric approach
is applicable here.

The main and most subtle part of Theorem \ref{main:locality} is contained in the assertion about the
perpendicularity of the mean curvature vector $ \mcv{F}(V, \cdot) $. If $ F $ is the integrand
associated with the Euclidean norm, perpendicularity of the mean curvature vector, for $ \| V\| $
almost all points, was proved by Brakke in~\cite[5.8]{Brakke1978} for a general integral varifold
$ V $ whose Euclidean first variation $ \delta V $ is representable by integration (see also
\cite[Remark 4.9]{Menne13}). One ingredient of Brakke's proof is the height decay proved in
\cite[5.7]{Brakke1978}. In the anisotropic setting, the quadratic flatness proved \cite{KolSan2}
provides an optimal quadratic height decay for varifolds satisfying the hypothesis of Theorem
\ref{main:locality}.  On the other hand, Brakke's proof employs the Euclidean structure in a crucial
way, in particular the Pythagorean theorem, to carry on some of the fundamental estimates (cf.\
Remark \ref{rem:perp:comparison_to_Brakke}). Obviously, the latter is an important missing
ingredient in the anisotropic setting. To circumvent this issue, our~proof involves the new idea
that one needs to \emph{locally} straighten the anisotropic structure around a point $ b $ through
a~linear map $ \varphi_T $, depending on the approximate tangent plane $ T = \Tan^k(\| V \|,b) $ of
$ \| V \| $; see~\ref{mr:def_varphi_T}. Concretely, this idea translates to studying the first
variation with respect to the integrand $ G $ of the pushforward $ (\varphi_T^{-1})_\# V $ of $ V $,
where $ G = \varphi_T^\# F $ is given by the pull-back of $ F $ by the linear map $ \varphi_T $. A
subtle technical part of this argument requires a special care to deal with the Lebesgue points of
the mean curvature vector of the pulled-back varifold $ (\varphi_T^{-1})_\# V $; see
Lemma~\ref{lem:pull_back:Lebesgue_points}. We remark that part of this argument actually holds for a
general $ \cnt{1} $ integrand of arbitrary codimension; cf.\ Theorem \ref{thm:mcv-orto}.

Relying on perpendicularity and quadratic flatness, we can prove that the distributional mean curvature $ \mcv{F}(V, \cdot) $ agrees $ \HM^n $ almost everywhere with the mean $ F $-curvature vectors of the $ \cnt{2} $-rectifiable cover of $ M $. This kind of result is known as \emph{locality of the mean curvature}
(see~\ref{thm:locality}), and it is proved by adapting the ideas from \cite{Ambrosio2003}, where locality in the Euclidean setting is proved by exploiting in a crucial way the quadratic tilt-excess decay by Sch\"atzle \cite{Schatzle04}. See also \cite{Schatzle09}  and \cite{Menne13} for more general locality theorems in the Euclidean setting. It might be interesting to remark that the anisotropic extension of the locality theorem still exploits a few peculiarities of
the anisotropic first variation in an essential way;
in~particular~\ref{rem:P_projection}.

Concerning the restriction to the unit-density layer in Theorem~\ref{main:locality} we remark that
perpendicularity of the \emph{Euclidean} mean curvature vector also works at higher density
points. This is a fundamental consequence of the Euclidean multivalued Lipschitz approximation,
where fast decay estimates are guaranteed by the monotonicity formula. However, in the anisotropic
case the situation becomes far more difficult, since monotonicity formulas are notoriously not
available, and the only available Lipschitz approximation theorem~\cite[\S{2.6}]{Allard1986} seems
to provide too weak estimates on the set not covered by the graph of the constructed Lipschitzian
map; see~\ref{rem:perp:bad_set_estimates}. Such a set is obviously empty if we restrict on the
unit-density layer as, by the aforementioned regularity theorem, this layer is almost covered by one
single $ \cnt{1, \alpha} $ manifold; see \ref{rem:perp:bad_set_estimates}. It remains an open problem to prove perpendicularity at the higher-density points.

\section{Preliminaries}
\label{sec:prelim}

\nsubsection{Notation}
\label{sec:notation}

\begin{miniremark}
    \label{mr:basic_notation}
    In~principle, but with some exceptions, we follow the notation of Federer;
    see~\cite[pp.~669--671]{Federer1969}. For varifolds we follow the notation of \cite{Allard1972}.
    
    We write $\natp$ for the set of positive integers but $A \without B$ for the set-theoretic
    difference.
    
    Whenever $L \in \Hom(X,Y)$ we write $L(x)$, $L x$, or $\langle x ,\, L \rangle$ to denote the
    value of~$L$ on~$x$. In~case $S \subseteq \R^{\adim}$ the closure in~$\R^{\adim}$ of~$S$ is
    $\Clos{S}$, the interior is denoted $\Int S$, and its topological boundary in~$\R^{\adim}$
    is~$\Bdry S$. The~characteristic function of $S \subseteq \R^{\adim}$ is $\CF{S} $. Given any
    set $X$ the identity map on~$X$ is the function $\id{X} : X \to X$.  We~write $\cball ar$ and
    $\oball ar$ for the closed and open ball of radius $0 < r < \infty$ and centre~$a$ (in the
    metric space that~$a$ belongs to, which should be clear from the context). If $X$ is a real
    vectorspace, $v \in X$, and $r \in \R$, the translation $\trans{v} : X \to X$ and dilation
    $\scale{r} : X \to X$ are defined by $\trans{v}(x) = v+x$ and $\scale{r}(x) = rx$ for $x \in X$.
    Whenever $X$ is a locally compact Hausdorff space we write $\ccspace{X}$ for the space
    consisting of all compactly supported continuous functions mapping $X$ to $\R$. Whenever $M$ is
    a~smooth submanifold of $\R^{\adim}$ we define
    \begin{displaymath}
        \VF(M) = \cnt{\infty}(M,\R^{\adim}) \cap \bigl\{
        g : \spt g \text{ is compact} ,\,
        g(x) \in \Tan(M,x) \text{ for } x \in M
        \bigr\} \,.
    \end{displaymath}
    Let $k,l \in \natp$ and $k \le \adim$. A set $\Sigma \subseteq \R^{\adim}$ is called
    \emph{countably $(\HM^k,k)$~rectifiable of class $\cnt{l}$} if there exists a~countable family
    $\mathcal{A}$ of~submanifolds of~$\R^{\adim}$ of~dimension~$k$ and class~$\cnt{l}$ such that
    $\HM^k(\Sigma \without \tbcup \mathcal{A}) = 0$. If, additionally, $\HM^{k}(\Sigma) < \infty$,
    we say that $\Sigma$ is~\emph{$(\HM^k,k)$~rectifiable of class $\cnt{l}$}. In~case $l=1$, we
    omit ``of class $\cnt{l}$''.

    Given $ \mu_1 $ and $ \mu_2 $ Radon measures, the derivative of $ \mu_1 $ with respect to
    $ \mu_2 $ is the $ \R $-valued $ \mu_2 $-measurable function defined by the formula
    \begin{displaymath}
        \mathbf{D}(\mu_1, \mu_2, x) = \lim_{r \to 0} \frac{\mu_1(\cball{x}{r})}{\mu_2(\cball{x}{r})}; 
    \end{displaymath}
    cf.~\cite[2.8 and 2.9]{Federer1969}. If $ \mu $ is a measure over a normed vectorspace $ X $
    and $ k \in \natp $ then
    \begin{displaymath}
        \Tan^k(\mu,a) 
    \end{displaymath}
    denotes the cone of $ (\mu, k) $ approximate tangent vectors at $ a $;
    cf. \cite[3.2.16]{Federer1969}. If $ X = \R^{\adim} $ and $ \mu = \HM^k \restrict \Sigma $,
    where $ \Sigma $ is $(\HM^k, k) $-rectifiable, then $ \Tan^k(\HM^k \restrict \Sigma, a) $ is a
    $ k $-dimensional vectorsubspace of $ \R^{\adim} $ for $ \HM^k $ a.e.\ $ a \in \Sigma $; cf.\
    \cite[3.2.19]{Federer1969}.
\end{miniremark}

\begin{miniremark}
    \label{mr: inner products and norms}
    Whenever $ X $ is a finite dimensional vector space endowed with an inner product $ \bullet $,
    then we denote by $|\cdot|$ the associated Euclidean norm. Given a map $A \in \End{X}$ we~define
    \begin{displaymath}
        |A| = \trace (A^* \circ A)^{1/2}\quad \textrm{and} \quad \|A\| = \sup \{|Ax| : x \in X ,\, |x| \le 1 \}. 
    \end{displaymath}
    Note that if $\dim \im A = 1$, then $|A| = |Ax| = \|A\|$ whenever $ x \in (\ker A)^\perp $ and
    $ |x|=1 $.
    
    We recall that an inner product $ \bullet $ on a finite dimensional vector space $ V $ induces
    an inner product $ \bullet $ on $ \tbwedge_{k}V $ for every $ k \geq 1 $; cf.~\cite[1.7.5 and
    1.7.6]{Federer1969}.
\end{miniremark}

\begin{miniremark}  
    For $k,m \in \natp$ with $0 \le k \le m$ we~denote by $ \grass mk $ the compact \emph{Grassmann
      manifold} of $ k $-dimensional planes in $ \R^m $. Following Almgren's convention, for
    $ T \in \grass mk $ we write $ \project{T} : \R^m \rightarrow \R^m $ to denote the orthogonal
    projection onto $ T $. The map
    \begin{displaymath}
        \bigl[ T \mapsto \project{T} \bigr] : \grass{m}{k} \to \End{\R^{m}}
    \end{displaymath}
    is a smooth embedding of $ \grass{m}{k} $ into $ \End{\R^m} $, whose image is the compact
    submanifold $ \pgrass{m}{k} $ of $ \End{\R^m} $,
    \begin{displaymath}
        \pgrass{m}{k} =	\End{\R^{m}} \cap \bigl\{ P : P \circ P = P ,\, P^* = P ,\, \trace P = k \bigr\} \,.
    \end{displaymath}
    As customary, we shall frequently identify $\grass mk$ with $\pgrass{m}{k}$. Recall
    (e.g.~\cite[Appendix A]{Philippis2018}) that
    \begin{multline}
        \Tan(\grass{m}{k}, T) = \End{\R^m} \cap \bigl\{
        L : L^\ast = L, \;
        \project{T} \circ L \circ \project{T} = \perpproject{T} \circ L \circ \perpproject{T} = 0
        \bigr\}
        \\ \text{for $ T \in \grass{m}{k} $} \,. 
    \end{multline}
\end{miniremark}

\begin{lemma}
    \label{mr:jacobian_multiplicative}
    Suppose $ V $ is an inner product space, $\dim V = n < \infty$, $ M, L \in \End{V} $,
    $ \dim \im L \leq k $ and $ S = \project{(\im L)} $.

    Then $ 	|\tbwedge_{k}( M \circ L)| = | \tbwedge_{k} (M \circ S) | \cdot | \tbwedge_{k} L | $.
\end{lemma}

\begin{proof}
    If $ \dim \im (M \circ L) < k $, then $ \dim \im (M \circ S) < k $ and both sides of the
    asserted equality equal zero. Henceforth, we assume $ \dim \im (M \circ L) = k $, in which case
    $ \dim \im (M \circ S) = \dim \im L = k $.  We choose an orthonormal basis $ v_1, \ldots, v_n $
    of $ V $ such that $ v_{1}, \ldots, v_k $ spans $ (\ker L)^\perp $ and set
    $ \xi = v_1 \wedge \ldots \wedge v_k $. Then \begin{displaymath} | \xi | = 1, \quad \langle
        \tbwedge_{k} L \xi, \tbwedge_{k} S \rangle = \tbwedge_{k} L \xi \quad \text{and} \quad |
        \tbwedge_{k} L | = | \tbwedge_{k} L \xi |.
    \end{displaymath}
    Defining $\eta = \tbwedge_{k} L \xi / |\tbwedge_{k} L \xi|$, we notice that 
    \begin{displaymath}
        \tbwedge_{k} (M \circ L)(\xi) = | \tbwedge_{k} L \xi |\,\tbwedge_{k} (M \circ S)(\eta)
    \end{displaymath} 
    and, since $\dim \im \tbwedge_{k} (M \circ L) = 1$, the conclusion follows from \ref{mr: inner
      products and norms}
\end{proof}

\begin{miniremark}
    \label{mr:taking_image_of_an_endomorphism}
    Let $R = \End{\R^{m}} \cap \bigl\{ A : \dim \im A = k \bigr\}$ and $f : R \to \pgrass{m}{k}$ be
    given by
    \begin{displaymath}
        f(A) = \eqproject{\im A} \quad \text{for $A \in R$} \,.
    \end{displaymath}
    Let $A \in R$ and choose $p \in \orthproj{m}{k}$ such that $\im p^* = \im A$. Note that
    \begin{displaymath}
        U = R \cap \bigl\{ B : \ker (B \circ p^*) = \{0\} \bigr\}
        = R \cap \bigl\{ B : \tbwedge_{k} (B \circ p^*) \ne 0 \bigr\}
    \end{displaymath}
    is an open neighbourhood of $A$ in~$R$. Observe
    \begin{displaymath}
        f(B) = B \circ p^* \circ \bigl( p \circ B^* \circ B \circ p^* \bigr)^{-1} \circ p \circ B^*
        \quad \text{for $B \in U$} \,.
    \end{displaymath}
    If follows that $ f | U $ is a smooth function. Hence,  \emph{$f$ is smooth and locally Lipschitzian}.
\end{miniremark}

\begin{miniremark}
    \label{mr:Lebesgue_points_and_taking_image}
    Let $X$ and $Y$ be finite dimensional normed vectorspaces, $U \subseteq \R^m$ be open, $\mu$ be
    a~Radon measure over~$U$, $g : U \to X$ be locally $\mu$-summable, and $f : \im g \to Y$ be
    locally Lipschitzian. Assume $g$ maps bounded sets into bounded sets and $x \in U$ is a
    $\mu$-Lebesgue point of $g$. One readily verifies that under these assumptions $x$ is also
    a~$\mu$-Lebesgue point of $f \circ g$.

    In~particular, suppose $R$ and $f$ are as in~\ref{mr:taking_image_of_an_endomorphism},
    $Y = X = \End{\R^{m}}$, $\tau : U \to \pgrass{m}{k}$ is $\mu$-measurable, $x \in U$ is a
    $\mu$-Lebesgue point of $\tau$, and $h : U \to X \cap \{ A : \ker A = \{0\} \}$ is
    continuous. Define $g : U \to X$ by $g(y) = h(y) \circ \tau(y)$ for $y \in U$, and notice that
    $ g $ is locally $\mu$-summable and has bounded image (since $\pgrass{m}{k}$ is compact)
    in~$R$. Then we conclude that $x$ is a~$\mu$-Lebesgue point of~$f \circ g$. Notice that
    \begin{displaymath}
        f(g(y)) = \project{h(y)[\im \tau(y)]} \quad \textrm{for $ y \in U $}  \,.
    \end{displaymath}
\end{miniremark}

\begin{miniremark}
    In the whole paper we assume $\phi$~is a~uniformly convex~$\cnt{2}$-norm on~$
    \R^{\adim}$. Uniform convexity implies that there exists an~\emph{ellipticity constant}
    $ \gamma(\phi) > 0 $ such that
    \begin{displaymath}
        \Der^2 \phi(u)(v,v) \ge \gamma(\phi) |v|^2
        \quad \text{for $ u \in \sphere{\vdim}$ and $ v \in \lin \{ u \}^{\perp} $} \,.
    \end{displaymath}
    For $l \in \nat$ we also define 
    \begin{gather}
        c_{l}(\phi) = \sup \bigl\{
        \|\Der^k \phi(\nu) \| : \nu \in \sphere{\vdim} ,\; k \in \integers ,\; 0 \le k \le l
        \bigr\} \,.
    \end{gather}
    Note that since $\phi$ is $\cnt{2}$-regular away from the origin the constant~$\gamma(\phi)$
    coincides with the constant named ``$\gamma$'' in~\cite[3.1(4)]{Allard1986}.
\end{miniremark}

\nsubsection{Varifolds and their mean curvature}
\label{sec:varifolds}

\renewcommand{\adim}{m}
\renewcommand{\vdim}{k}
\renewcommand{\codim}{m-k}

\begin{miniremark}
    \label{mr:setup_for_integrands}
    Let $0 \le \vdim \le \adim$ be integers and $ \Omega \subseteq \R^{\adim} $ is open. If
    $j \in \natp \cup \{0\}$, then a~\emph{$k$-dimensional $ \cnt{j} $-integrand on $ \Omega $} is a
    non-negative function $F : \Omega \times \grass{\adim}{k} \to \R$ of class $ \cnt{j} $. A
    $ k $-dimensional integrand $ F $ is called \emph{autonomous} if it does not depend on the space
    variable; i.e.\ $ F(x,T) = F(T) $ for $ (x,T) \in \Omega \times \grass{\adim}{k} $.
    
    In this section we suppose
    \begin{gather}
        U, \Omega \subseteq \R^{\adim} \quad \text{are open} \,,
        \quad
        0 \le k \le \adim \quad \text{integer} \,,
        \\
        \pi : \R^{\adim} \times \End{\R^{\adim}} \to \R^{\adim} \,,
        \quad
        \sigma : \R^{\adim} \times \End{\R^{\adim}} \to \End{\R^{\adim}} \,,
        \\
        \pi(x,A) = x 
        \quad \text{and} \quad
        \sigma(x,A) = A
        \quad \text{for $x \in \R^{\adim}$ and $A \in \End{\R^{\adim}}$} \,.
    \end{gather}
\end{miniremark}

\begin{definition}
    \label{def:pull-back-integrand}
    Suppose $\psi : U \to \Omega$ is a smooth diffeomorphism, and
    $ \alpha : \Omega \times \grass{\adim}{k} \rightarrow \R $ is continuous. We define the
    \emph{pull-back of $ \alpha $} by
    \begin{displaymath}
        \psi^{\#}\alpha(y,T) =
        \alpha \bigl( \psi(y), \uD \psi(y) \lIm T \rIm \bigr) \,|\tbwedge_{\vdim}( \uD \psi(y) \circ \project{T})|
        \quad \textrm{for $(y,T) \in U \times \grass{\adim}{k}$} \,.
    \end{displaymath}
\end{definition}

\begin{remark}
    If $ \alpha : \Omega \times \grass{\adim}{k} \rightarrow \R $ is a function of class $ \cnt{j} $
    then $ \varphi^\# \alpha $ is of class $ \cnt{j} $.
\end{remark}

\begin{remark}
    Suppose $ \varphi : U \rightarrow W $ and $ \psi : W \rightarrow \Omega $ are smooth
    diffeomorphisms. Given $ (y, T) \in U \times \grass{\adim}{k} $, we employ
    Lemma~\ref{mr:jacobian_multiplicative} with $ M = \Der \psi(\varphi(y)) $ and
    $ L = \Der \varphi(y) \circ \project{T} $ to verify that
    $ \varphi^\# (\psi^\#\alpha)(y,T) = (\psi \circ \varphi)^\# \alpha (y,T) $.
\end{remark}

\begin{miniremark}[\protect{\cite{Allard1972}}]
    \label{def:varifold}
    A~$k$-dimensional \emph{varifold} in $\Omega$ is a Radon measure over
    $\Omega \times \grass{\adim}{k}$. The space of all $k$-dimensional varifolds in~$\Omega$ is
    denoted $\Var{k}(\Omega)$. This space is endowed with the weak topology;
    cf.~\cite[2.5.19]{Federer1969}. \emph{Rectifiable} and \emph{integral} varifolds are defined as
    in~\cite[3.5]{Allard1972} and corresponding spaces are denoted $\RVar{k}(\Omega)$ and
    $\IVar{k}(\Omega)$ respectively. Whenever $M$ is a $k$-dimensional properly embedded submanifold
    of~$\Omega$ of class~$\cnt{1}$ we associate to it the varifold $\var{k}(M) \in \IVar{k}(\Omega)$
    defined by the formula
    \begin{displaymath}
    	\var{k}(M)(A) = \HM^k \bigl( M \cap \{ x : (x,\Tan(M,x)) \in A \} \bigr)
    	\quad \text{for $A \subseteq \Omega \times \grass{\adim}{k}$} \,.
    \end{displaymath}
    We also recall that if $ V \in \Var{k}(\Omega) $ then for $ \| V \| $ almost all
    $ x \in \Omega $ there exists a probability measure $ V^{(x)} $ on $ \grass{\adim}{k} $ such that
    \begin{displaymath}
        \textint{}{} \alpha(x,S) \ud V(x,S) = \textint{}{} \textint{}{} \alpha (x,S) \ud V^{(x)}(S) \ud \| V \|(x) 
    \end{displaymath}
    whenever $ \alpha \in \ccspace{\Omega \times \grass{\adim}{k}} $. 
\end{miniremark}

\begin{miniremark}
    If $ F $ is a continuous $ k $-dimensional integrand and $ V \in \Var{k}(\Omega) $, it is
    convenient to introduce the following functional
    \begin{displaymath}
        V_F(\alpha) = \textint{}{} \alpha(x,T)F(x,T) \ud V(x,T)
        \quad \text{for $ \alpha \in \ccspace{\Omega \times \grass{\adim}{k}} $} \,.
    \end{displaymath}
\end{miniremark}

\begin{definition}
    Suppose $ \psi : U \rightarrow \Omega $ is a smooth diffeomorphism and $V$ is
    a~$k$-di\-men\-sio\-nal varifold in~$U$. We define the \emph{push-forward, $ \psi_\# V $, of
      $ V $} to be the $ k $-dimensional varifold in $ \Omega $ defined by
    \begin{displaymath}
        \psi_\# V(\alpha) = V(\psi^\#\alpha)
        \quad \text{for every $ \alpha \in \ccspace{U \times \grass{\adim}{k}}$} \,.
    \end{displaymath}
\end{definition}

\begin{miniremark}[\protect{cf.~\cite[2.1(4)]{Allard1986} and~\cite[Appendix~A]{Philippis2018}}]
    \label{rem:BF_as_derivative}
    Suppose $ F $ is a $\vdim$-dimensional $ \cnt{1} $-integrand on $ \Omega $. 
  
    For $L \in \End{\R^{\adim}}$ and define $L_{t} \in \End{\R^{\adim}}$ by
    $L_{t} = \id{\R^{\adim}} + t L$ for $t \in \R$. Referring to~\cite[2.1]{Allard1986}
    and~\cite[Appendix~A]{Philippis2018} we~recall that
    \begin{equation}
        B_{F}(x,T) \bullet L
        = \left. \tfrac{\ud}{\ud t} \right|_{t=0}
        F(x, L_{t} \lIm T \rIm) | \tbwedge_{\vdim} (L_{t} \circ \project{T}) |
    \end{equation}
    for $x \in \Omega$ and $T \in \grass{\adim}{\vdim}$, where $B_{F}(x,T) \in \End{\R^{\adim}}$ is
    characterized by
    \begin{equation*}
        B_{F}(x,T) \bullet L
        = F(x,T) \project{T} \bullet L
        + \bigl\langle
        \perpproject{T} \circ L \circ \project{T}
        + \transpose{(\perpproject{T} \circ L \circ \project{T})} ,\, \uD F(x,T) \circ \transpose{\sigma}
        \bigr\rangle
    \end{equation*}
    for $L \in \End{\R^{\adim}}$. Henceforth, if $V$ is a~$k$-di\-men\-sio\-nal varifold in~$\Omega$,
    $ g \in \VF(\Omega) $, $ h $ is the flow of $ g $ and $ G \subseteq \Omega $ such that
    $ \spt g \subseteq G $ and $ \| V \|(G) < \infty $, then
    \begin{multline}
        \left. \tfrac{\ud}{\ud t} \right|_{t=0} h_{t\#}\bigl(V \restrict G \times \grass{\adim}{\vdim}\bigr)(F)
        \\
        =  \left.\tfrac{\ud}{\ud t} \right|_{t=0} \textint{G \times \grass{\adim}{\vdim}}{}
        F\bigl(h_t(x), \uD h_t(x) \lIm T \rIm \bigr)
        \,|\tbwedge_{\vdim} \uD h_t(x) \circ \project{T}| \ud V(x,T)
        \\
        = \textint{}{} \bigl\langle g(x) ,\, \uD F(x,T) \circ \transpose{\pi} \bigr\rangle
        + B_{F}(x,T) \bullet \uD g(x) \ud V(x,T) \,.
    \end{multline}
    
    As customary we define \emph{first variation of $ V $ with respect to $ F $} as the linear map
    $ \delta_F V : \VF(\Omega) \rightarrow \R $ defined by
    \begin{multline}
        \delta_F V(g)
        = \left. \tfrac{\ud}{\ud t} \right|_{t=0} h_{t\#}\bigl(V \restrict G \times \grass{\adim}{\vdim}\bigr)(F)
        \\ 
        = \textint{}{} \bigl\langle g(x) ,\, \uD F(x,T) \circ \transpose{\pi} \bigr\rangle
        + B_{F}(x,T) \bullet \uD g(x) \ud V(x,T)
    \end{multline}
    whenever $ g $, $ h $ and $ G $ are as above. The \emph{total
   	variation} of~$\delta_FV$ is the largest Borel regular measure
   $\|\delta_F V\|$ over $\Omega$ determined by the requirement that
   \begin{multline}
   	\|\delta_FV\|(U) = \sup \bigl\{
   	\delta_FV(g)
   	: g \in \VF(\Omega) ,\,
   	\spt g \subseteq U ,\,
   	|g| \le 1
   	\bigr\}
   	\\
   	\text{whenever $U \subseteq \Omega$ is open} \,.
   \end{multline}
\end{miniremark}

\begin{miniremark}
    \label{rem:perp:first_variation_pull_back}
    Suppose $\psi : U \rightarrow \Omega $ is a smooth~diffeomorphism, $ \varphi = \psi^{-1} $,
    $ G = \varphi^\# F $ and $ V \in \Var{\vdim}(\Omega)$. For $g \in \VF(\Omega)$, let $h$ be the
    flow of~$g$ and $ \Omega^{\prime}$ is an open set with $ \spt g \subseteq \Omega^{\prime}$, and
    whose closure is a compact subset of $ \Omega $.
    \begin{multline}
        h_{t\#} \bigl( (\psi_\# V) \restrict \Omega' \times \grass{\adim}{k} \bigr) ( G )
        = \bigl( (\psi_\# V) \restrict \Omega' \times \grass{\adim}{k}\bigr) \bigl( (\varphi \circ h_t)^\# F \bigr)
        \\
        = (\varphi \circ h_t \circ \psi)_\#\bigl(
        V \restrict \bigl( \psi^{-1} \lIm \Omega' \rIm \times \grass{\adim}{k}\bigr)
        \bigr)(F) 
    \end{multline}
    so, since $\varphi \circ h_0 \circ \psi = \id{U}$, differentiating in~$t$ at~$0$ yields
    \begin{displaymath}
        \delta_{G} (\psi_\# V)(g)
        = \delta_F (V) ( \langle g \circ \psi ,\, \uD \varphi \circ \psi \rangle) \,.
    \end{displaymath}
\end{miniremark}

\begin{definition}
    \label{def:related_projections}
    For $x \in \Omega$ and $T \in \grass{\adim}{\vdim}$ set
    \begin{displaymath}
        P_{F}(x,T) = B_{F}(x,T)^* / F(x,T)
        \quad \text{and} \quad
        Q_{F}(x,T) = \id{\R^{\adim}} - P_{F}(x,T) \,.
    \end{displaymath}
\end{definition}

\begin{remark}
    \label{rem:P_projection}
    Let $x \in \Omega$ and $T \in \grass{\adim}{\vdim}$. Straightforward computations show that
    \begin{gather}
        \project{T} \circ B_{F}(x,T) = F(x,T) \project{T}
        \quad \text{and} \quad
        \perpproject{T} \circ B_{F}(x,T) = \perpproject{T} \circ B_{F}(x,T) \circ \project{T} \,;
        \\
        \text{hence,} \quad
        P_{F}(x,T) = \project{T} + \project{T} \circ P_{F}(x,T) \circ \perpproject{T}
        = \project{T} +   P_{F}(x,T) \circ \perpproject{T} \,,
        \\
        P_{F}(x,T) \circ P_{F}(x,T) = P_{F}(x,T) \,,
        \quad 
        \im P_{F}(x,T) = T \,,
        \quad \text{and} \quad
        T \cap \ker P_F(x,T) = \{0\}\,.
    \end{gather}
    In~particular, we obtain that $P_F(T)$ and $Q_F(T)$ are projections which are orthogonal only if
    $\uD F(x,T) \circ \sigma^{\ast} = 0$.
\end{remark}

\begin{miniremark}
    \label{rem:integral_representation_of_first_variation}
    As in~\cite[4.3]{Allard1972} if $\|\delta_FV\|$ is Radon, we can represent $\delta_FV$ as
    \begin{multline}
        \delta_F V(g)
        = - \textint{}{} h(x) \bullet g(x) \ud \| V_F \|(x)
        + \textint{}{} \eta(x) \bullet g(x) \ud \| \delta_F V \|_{\mathrm{sing}} (x)
        \\
        = \textint{}{} \eta(x)\bullet g(x) \ud \| \delta_F V \|(x)
        \quad
        \text{for $g \in \VF(\Omega)$}\,,
    \end{multline}
    where $h$ is a~$\|V\|$~measurable $\R^{\adim}$-valued function, $\eta$ is
    some~$\| \delta_F V \|$~measurable $\sphere{\adim-1}$-valued function and
    $\| \delta_F V \|_{\mathrm{sing}}$ is the singular part of~$\| \delta_F V \|$ with respect to
    $\|V\|$.
    
    Employing the theory of symmetric derivation (see \cite[2.8.18, 2.9]{Federer1969} or
    \cite[Theorem 2.22]{AmbrosioFuscoPallara}) we see that the formula
    \begin{displaymath}
        \mathbf{D} \bigl( \| \delta_F V \|, \| V \|, x\bigr)
        = \lim_{r \to 0}\tfrac{ \| \delta_F V \|(\cball{x}{r})}{\| V \|(\cball{x}{r})}
    \end{displaymath}
    defines a real-valued $ \| V \| $-measurable function
    $\mathbf{D} \bigl( \| \delta_F V \|, \| V \|, \cdot \bigr)$ such that
    \begin{gather}
        \| \delta_F V \|_{\mathrm{sing}}
        = \| \delta_F V \| \restrict \bigl\{x :  \mathbf{D} \bigl( \| \delta_F V \|, \| V \|, x\bigr) = \infty \bigr\}
        \\
        \text{and} \quad
        \| \delta_F V \| = \mathbf{D} \bigl( \| \delta_F V \|, \| V \|, \cdot\bigr)\cdot \| V \|
        + \| \delta_F V \|_{\mathrm{sing}} \,.
    \end{gather}
\end{miniremark}

\begin{definition}
    \label{def:wmcv}
    Suppose $V$ is a~$k$-dimensional varifold in~$\Omega$ and $ F $ is a $ \cnt{1} $-integrand on
    $ \Omega $ such that $ \| \delta_F V \| $ is a Radon measure. We denote any
    $ \| V \| $-measurable function $ h $ as in \ref{rem:integral_representation_of_first_variation}
    by $ \mcv{F}(V, \cdot) $. Moreover, we denote any $ \| \delta_F V \| $-measurable function
    $ \eta $ as in \ref{rem:integral_representation_of_first_variation} by
    $ \boldsymbol{\eta}(V, \cdot) $, we set
    \begin{displaymath}
        \wmcv{F}(V,\cdot) =  \mcv{F}(V, \cdot)\, \textint{}{} F(\cdot, T) \ud V^{(\cdot)}(T)
    \end{displaymath}
    and we refer to $ \wmcv{F}(V, \cdot) $ as \emph{mean $ F $-curvature vector of $ V $}.
\end{definition}

\renewcommand{\adim}{{n+1}}
\renewcommand{\vdim}{n}
\renewcommand{\codim}{1}

\begin{miniremark}
    \label{mr:anisotropic_mean_curvature_vector}
    Let $ \nu : \grass{\adim}{\vdim} \to \sphere{\vdim} $ be a Borel map such that
    $ \nu(S) \in S^\perp \cap \sphere{\vdim} $ for each $ S \in \grass{\adim}{\vdim} $.
    As in~\cite[3.1]{Allard1986} we~associate an integrand~$F : \grass{\adim}{\vdim} \to \R$ with
    a~norm~$\phi$ by requiring that
    \begin{displaymath}
        F(T) = \phi(\nu(T))
        \quad \text{whenever $ T \in \grass{\adim}{\vdim}$} \,.
    \end{displaymath}
    Recalling~\cite[2.16]{DRKS2020ARMA} we~get
    \begin{equation}
        \langle u,\, B_{F}(T) \rangle
        = \phi(\nu(T)) u - \bigl( \grad \phi(\nu) \bullet u \bigr)\, \nu(T)
    \end{equation}
    whenever $T \in \grass{\adim}{\vdim}$ and $ u \in \R^{\adim} $.
\end{miniremark}

\section{Controlling tilt with height}

\renewcommand{\adim}{{n+1}}
\renewcommand{\vdim}{n}
\renewcommand{\codim}{1}

In this section $F$ is the integrand associated to the uniformly convex norm $\phi$ as
in~\ref{mr:anisotropic_mean_curvature_vector}.

\begin{miniremark}
    \label{mr:basic_relation_in_codimension_one}
    Recall~\cite[8.9(3)]{Allard1972} and observe
    \begin{multline}
        \tfrac{1}{\sqrt{2}} |u - \sgn(u \bullet v) v| 
        = \bigl( 1 - |u \bullet v| \bigr)^{1/2}
        = \| S \circ T^{\perp} \|
        = \| S - T \|
        \\
        \text{whenever
          $S,T \in \pgrass{\adim}{\vdim}$,
          $u,v \in \sphere{\vdim}$,
          $u \in \ker S$,
          and $v \in \ker T$}\,.
    \end{multline}
\end{miniremark}

\begin{lemma}
    \label{lem:PQ_estimates}
    Let $S,T \in \grass{\adim}{\vdim}$. There exists $1 \le \Gamma < \infty$ depending only
    on~$\gamma(\phi)$ and $c_{2}(\phi)$ such that
    \begin{gather}
        \|Q_{F}(T)\| \le \Gamma \,,
        \quad
        \| P_{F}(S) \circ Q_{F}(T) \|
        + \| P_{F}(S) - P_{F}(T) \|
        \le \Gamma \| \project{S} - \project{T} \| \,,
        \\
        \label{eq:sac}
        \text{and} \quad
        P_{F}(S)^* \bullet Q_{F}(T) \ge \Gamma^{-1} \| \project{S} - \project{T} \|^2 \,.
    \end{gather}
\end{lemma}

\begin{proof}
    Let $\nu,\eta \in \sphere{\vdim}$ be such that $\nu \bullet \eta \ge 0$, $\nu \in T^{\perp}$,
    and $\eta \in S^{\perp}$. Note that the first estimate is trivially true because $Q_{F}(T)$ is
    defined using $\phi(\nu)$ and $\uD \phi(\nu)$, the sphere $\sphere{\vdim}$ is compact, and
    $\phi$ is of class~$\cnt{2}$ away from the origin. Let $\varphi : \R \to \sphere{\vdim}$ be a
    geodesic such that $\varphi(0) = \nu$ and $\varphi(l) = \eta$, where
    $l = \arccos(\eta \bullet \nu)$. Recall~\ref{mr:basic_relation_in_codimension_one} and note that
    since $\eta \bullet \nu \ge 0$ there holds
    \begin{displaymath}
        \sqrt{2} \| \project{S} - \project{T} \|
        = |\eta - \nu|
        \le l
        \le \tfrac{\pi \sqrt{2}}{4} |\eta - \nu|
        = \tfrac{\pi}{2} \| \project{S} - \project{T} \| \,.
    \end{displaymath}
    Recall the definition of $\beta$ from~\ref{mr:basic_notation} and define
    $f,g : \R \to \End{\R^{\adim}}$ by
    \begin{gather}
        f(t) = P_{F}( \ker \beta(\varphi(t)) ) \circ Q_{F}(T)
        \quad \text{and} \quad
        g(t) = P_{F}(\ker \beta(\varphi(t))) - P_{F}(T)
        \quad \text{for $t \in \R$} \,.
    \end{gather}
    Clearly $f(0) = 0 = g(0)$ so
    \begin{gather}
        \| P_{F}(S) \circ Q_{F}(T) \|
        = \| f(l) - f(0) \|
        \le \tfrac{\pi \sqrt{2}}{4} \Lip f\, |\eta - \nu|
        \le \tfrac{\pi}{2} \Lip f\, \| \project{S} - \project{T} \|
        \\
        \text{and similarly} \quad
        \| P_{F}(S) - P_{F}(T) \|
        \le \tfrac{\pi}{2} \Lip g\, \| \project{S} - \project{T} \| \,.
    \end{gather}
    Since $\Lip f$ and $\Lip g$ are both controlled by $c_2(\phi)$ we obtain the second estimate.
    To~get the third estimate use~\cite[3.2(6)]{Allard1986} together
    with~\ref{mr:basic_relation_in_codimension_one} and write
    \begin{multline}
        \phi(\nu) \phi(\eta) P_{F}(S)^* \bullet Q_{F}(T)
        = \phi(\nu) \phi(\eta) P_{F}^*(S) \nu \bullet Q_{F}(T) \nu
        \\
        = \phi(\nu) \phi(\eta) - \bigl( \nu \bullet \grad \phi(\eta) \bigr) \bigl( \eta \bullet \grad \phi(\nu) \bigr)
        \ge \gamma ( 1 - |\eta \bullet \nu| ) = \gamma \| \project{S} - \project{T} \|^2 \,.
    \end{multline}
    This gives a lower bound with constant depending on $c_0(\phi)$ and $\gamma(\phi)$.
\end{proof}

\begin{remark}
    Lemma~\ref{lem:PQ_estimates} combines~\cite[3.2]{Allard1986}
    and~\cite[Theorem~1.3]{Philippis2018}. Observe that
    condition~\ref{lem:PQ_estimates}\eqref{eq:sac} says exactly that $F$ satisfies the \emph{scalar
      atomic condition} (SAC) introduced by De~Rosa and Tione;
    see~\cite[Definition~3.3]{DeRosaTione2022}.
\end{remark}

\begin{lemma}[\protect{Caccioppoli-type inequality}]
    \label{lem:tilt-height}
    Assume
    \begin{gather}
        \Omega \subseteq \R^{\adim} \quad \text{is open} \,,
        \quad
        c \in \Omega \,,
        \quad
        T \in \grass{\adim}{\vdim} \,,
        \quad
        V \in \Var{\vdim}(\Omega) \,,
        \\
        \text{$F$ is associated to the norm~$\phi$ as in~\ref{mr:anisotropic_mean_curvature_vector}} \,,
        \quad
        \|\delta_F V\| \quad \text{is Radon} \,,
        \\
        \zeta : \Omega \to \R \cap \{ t : 0 \le t \le 1 \} \quad \text{is~smooth} \,,
        \quad
        \spt \zeta \quad \text{is compact} \,.
    \end{gather}
    Set $\kappa = \sup\{ | \grad \zeta(x) | : x \in \Omega \}$. There exists $0
    < \Gamma < \infty$ depending only on $\gamma(\phi)$ and~$c(\phi)$ such that
    \begin{multline}
        \textint{}{}
        \zeta^2(z)
        \| \project{S} - \project{T} \|^2 \ud V_F(z,S)
        \le \Gamma \Bigl(
        \kappa^2 \textint{\spt \zeta}{}
        |\perpproject{T}(z-c)|^2 \ud \|V_F\|(z)
        \\
        + \textint{}{} \zeta^2(z)
        |\perpproject{T}(z-c)| \ud \|\delta_F V\|(z)
        \Bigr) \,.
    \end{multline}
\end{lemma}

\begin{proof}
    Let
    $\Delta = \Gamma_{\text{\ref{lem:PQ_estimates}}}(\gamma(\phi), c_{2}(\phi))$ be given
    by~\ref{lem:PQ_estimates}. Define $g \in \VF(U)$ by~setting
    \begin{displaymath}
        g(z) = \zeta(z)^2 Q_{F}(T) (z-c)
        \quad \text{for $x \in U$} \,.
    \end{displaymath}
    Note that, since $Q_{F}(T) = Q_{F}(T) \circ \perpproject{T}$ (by~\ref{rem:P_projection}), we~get
    for $x \in U$ and $S \in \grass{\adim}{\vdim}$
    \begin{multline}
        \bigl| P_{F}(S)^* \grad \zeta(x) \bullet Q_{F}(T)(z-c) \bigr|
        \\
        = \bigl| \grad \zeta(x) \bullet P_{F}(S) \circ Q_{F}(T) \circ \perpproject{T} (z-c) \bigr|
        \\
        \le |\grad \zeta(x)| \cdot \| P_{F}(S) \circ Q_{F}(T) \| \cdot |\perpproject{T} (z-c)|
        \,;
    \end{multline}
    hence, employing H{\"o}lder's inequality and recalling~\ref{lem:PQ_estimates} we get
    \begin{multline}
        \left| \textint{}{} \zeta(x)
            P_{F}(S)^* \grad \zeta(x) \bullet Q_{F}(T)(z-c)
            \ud V_{F}(x,S)
        \right|
        \\
        \le \kappa \Delta \left( \textint{\spt \zeta}{} |\perpproject{T}(z-c)|^2 \ud \| V_{F} \| (x) \right)^{1/2}
        \left( \textint{}{} \zeta(x)^2 \| \project{S} -\project{T} \|^2 \ud V_F(x,S) \right)^{1/2} 
        \,.
    \end{multline}
    We have
    \begin{multline}
        \textint{}{} \zeta(x)^2 \boldsymbol{\eta}_{F}(V;x) \bullet Q_{F}(T) \circ \perpproject{T}(z-c) \ud \|\delta_F V\|(x)
        = \delta_FV(g)
        \\
        = 2 \textint{}{} \zeta(x) P_{F}(S)^* \grad \zeta(x) \bullet Q_{F}(T)(z-c) \ud V_{F}(x,S)
        \\
        + \textint{}{} \zeta(x)^2 P_{F}(S)^* \bullet Q_{F}(T) \ud V_{F}(x,S) \,;
    \end{multline}
    thus,
    \begin{multline}
        \label{eq:th:after-holder}
        \Delta^{-1} \textint{}{} \zeta(x)^2 \| \project{S} - \project{T} \|^2 \ud V_{F}(x,S)
        \le \textint{}{} \zeta(x)^2 P_{F}(S)^* \bullet Q_{F}(T) \ud V_{F}(x,S)
        \\
        \le 2 \kappa \Delta \left( \textint{\spt \zeta}{} |\perpproject{T}(z-c)|^2 \ud \| V_F \| (x) \right)^{1/2}
        \left( \textint{}{}
            \zeta(x)^2 \| \project{S} - \project{T} \|^2 \ud V_{F}(x,S)
        \right)^{1/2}
        \\
        + \Delta \textint{}{} \zeta(x)^2 |\perpproject{T}(z-c)| \ud \| \delta_F V \| (x) 
        \,.
    \end{multline}
    Now, there are two possibilities. Either
    \begin{displaymath}
        \Delta \textint{}{} \zeta(x)^2 |\perpproject{T}(z-c)| \ud \| \delta_F V \| (x)
        \le \tfrac 12 \Delta^{-1}\textint{}{} \zeta(x)^2 \| \project{S} - \project{T} \|^2 \ud V_{F}(x,S) \,,
    \end{displaymath}
    which yields
    \begin{multline}
        \textint{}{} \zeta(x)^2 \| \project{S} - \project{T} \|^2 \ud V_{F}(x,S)
        \\
        \le 4 \kappa \Delta^2
        \left(
            \textint{\spt \zeta}{} |\perpproject{T}(z-c)|^2 \ud \| V_F \| (x)
        \right)^{1/2}
        \left( 
            \textint{}{} \zeta(x)^2  \| \project{S} - \project{T} \|^2 \ud V_{F}(x,S)
        \right)^{1/2}
    \end{multline}
    and, consequently,
    \begin{displaymath}
        \textint{}{} \zeta(x)^2 \| \project{S} - \project{T} \|^2 \ud V_{F}(x,S)
        \le 16 \kappa^2 \Delta^4
        \textint{\spt \zeta}{} |\perpproject{T}(z-c)|^2 \ud \| V_F \| (x) \,;
    \end{displaymath}
    or
    \begin{displaymath}
        \Delta \textint{}{} \zeta(x)^2 |\perpproject{T}(z-c)| \ud \| \delta_F V \| (x)
        \ge \tfrac 12 \Delta^{-1}\textint{}{} \zeta(x)^2 \| \project{S} - \project{T} \|^2 \ud V_{F}(x,S) \,.
        \qedhere
    \end{displaymath}
\end{proof}

\begin{remark}
    The proof of an analogous fact to~\ref{lem:tilt-height} in~\cite[3.3]{Allard1986} deals with
    infinite cylinders. Above we proved the statement more in the spirit of~\cite[8.13]{Allard1972}.
    Compare also~\cite[Proposition~4.3]{DeRosaTione2022} for the anisotropic case in arbitrary
    codimension.
\end{remark}

\begin{corollary}
    \label{cor:tilt-decay}
    Suppose
    \begin{gather}
        \Omega \subseteq \R^{\adim} \text{ is open} \,,
        \quad
        c \in \Omega \,,
        \quad
        T \in \grass{\adim}{\vdim} \,,
        \quad
        V \in \Var{\vdim}(\Omega) \,,
        \\
        \text{$F$ is associated to the norm~$\phi$ as in~\ref{mr:anisotropic_mean_curvature_vector}} \,,
        \\
        \label{eq:ccpp:singular_part}
        \density^{\vdim}(\|V\|, c) <  \infty \,,
        \quad
        \|\delta_F V\|_{\mathrm{sing}} = 0 \,,
        \quad
        \mcv{F}(V,\cdot) \in \Lploc{2}(\|V\|,\R^{\adim}) \,,
        \\
        \label{eq:ccpp:mean_curvature}
        \limsup_{r \to 0^{+}} r^{-\vdim}
        \textint{\cball cr}{}
        |\mcv{F}(V,z)|^2 \ud \|V_F\|(z) < \infty \,,
        \\
        \label{eq:ccpp:height_decay}
        \limsup_{r \to 0^{+}} r^{-\vdim-4}
        \textint{\cball{c}{r}}{} |\perpproject{T}(x-c)|^2 \ud \|V\|(x) < \infty
        \,.
    \end{gather}
    Then
    \begin{displaymath}
        \limsup_{r \to 0^{+}} r^{-\vdim-2} \textint{\cball{c}{r} \times \pgrass{\adim}{\vdim}}{}
        \| \project{S} - \project{T} \|^2 \ud V(z,S) < \infty \,.
    \end{displaymath}
\end{corollary}

\begin{proof}
    For $0 < r < \dist(c,\R^{\adim} \without \Omega)$ we
    apply~\ref{lem:tilt-height} with $\zeta$ satisfying
    \begin{gather}
        \spt \zeta \subseteq \oball cr \,,
        \quad
        \zeta(z) = 1 \quad \text{for $z \in \cball{c}{r/2}$} \,,
        \\
        \text{and} \quad
        \|\uD \zeta(z)\| \le 3/r \quad \text{for $z \in \Omega$} \,;
    \end{gather}
    this yields the estimate
    \begin{multline}
        r^{-\vdim-2} \textint{\cball{c}{r/2} \times \grass{\adim}{\vdim}}{}
        \| \project{S} - \project{T} \|^2 \ud V(z,S)
        \\
        \le \Gamma_{\text{\ref{lem:tilt-height}}}(\phi) \Bigl(
        9 r^{-\vdim-4} \textint{\cball cr}{}
        |\perpproject{T}(z-c)|^2 \ud \|V\|(z)
        \\
        + r^{-\vdim-2} \textint{\cball cr}{}
        |\perpproject{T}(z-c)| \ud \|\delta_{F}V\|(z)
        \Bigr)
        \,.
    \end{multline}
    The first term on the right-hand side stays bounded as~$r \to 0^{+}$
    by~\eqref{eq:ccpp:height_decay}. We~shall deal with the second one.
    Recalling~\eqref{eq:ccpp:singular_part} we get
    \begin{displaymath}
        \textint{\cball cr}{}
        |\perpproject{T}(z-c)| \ud \|\delta_{F}V\|(z)
        = \textint{\cball cr}{}
        |\perpproject{T}(z-c)| \cdot | \mcv{F}(V,z)| \ud \|V_F\|(z) \,.
    \end{displaymath}
    Applying H{\"o}lder's inequality
    \begin{multline}
        r^{-\vdim-2} \textint{\cball cr}{}
        |\perpproject{T}(z-c)| \cdot | \mcv{F}(V,z)| \ud \|V_F\|(z)
        \\
        \le \bigl( r^{-\vdim-4} \textint{\cball cr}{}
        |\perpproject{T}(z-c)|^2 \ud \|V_F\|(z) \bigr)^{1/2}
        \bigl( r^{-\vdim} \textint{\cball cr}{}
        | \mcv{F}(V,z)|^2 \ud \|V_F\|(z) \bigr)^{1/2} \,.
    \end{multline}
    By~\eqref{eq:ccpp:mean_curvature} and~\eqref{eq:ccpp:height_decay} this also stays bounded
    as~$r \to 0^{+}$.
\end{proof}

\section{First variation of a push-forward of a varifold}

\renewcommand{\adim}{m}
\renewcommand{\vdim}{k}
\renewcommand{\codim}{m-k}

In this section we prove some basic facts on the first variation of the push-forward of a~varifold.

\begin{miniremark}\label{push-forward basic}
    We assume 
    \begin{enumerate}
    \item $ U, \Omega \subseteq \R^m $ open subsets, 
    \item $V \in \RVar{\vdim}(U)$ with $ \density^{\vdim}(\| V \|,y) \geq 1 $ for $ \| V \| $ a.e.\ $ y \in U $, 
    \item $ \psi : U \rightarrow \Omega $ is a smooth diffeomorphism, $ \varphi = \psi^{-1} $ and $ W = \psi_\# V $,
    \item $F$ is a~$\vdim$-dimensional $ \cnt{1} $ integrand in~$ U $ and $ G = \varphi^\# F $, 
    \item $ \Sigma = U \cap \bigl\{ y : 0< \density^{\vdim}(\| V \|,y) < \infty \bigr\} $.
    \end{enumerate}

    Firstly, we observe that $ \Sigma $ is a Borel countably $(\HM^{\vdim}, {\vdim}) $-rectifiable set that
    meets each compact subset of $ U $ on a set of finite $ \HM^{\vdim} $ measure by \cite[2.8(5),
    3.5(1)]{Allard1972}. We can apply the area formula \cite[3.2.20]{Federer1969} to see that
    \begin{multline}
        \| W \|(S) = \textint{\psi^{-1} \lIm S \rIm}{}
        |\tbwedge_{\vdim} \uD \psi(y) \circ \Tan^{\vdim}(\| V \|,y)|
        \ud \| V \|(y)
        \\
        = \textint{\psi^{-1} \lIm S \rIm \cap \Sigma}{}
        |\tbwedge_{\vdim} \uD \psi(y) \circ \Tan^{\vdim}(\| V \|,y)| \, \density^{\vdim}(\| V \|,y)
        \ud \HM^{\vdim}(y)
        \\
        = \textint{S \cap \psi \lIm \Sigma \rIm}{}
        \density^{\vdim}(\| V \|,\psi^{-1}(x))
        \ud \HM^{\vdim}(x)
        \quad \text{whenever $ S \subseteq \Omega $ is a Borel set} \,.
    \end{multline}
    In particular,  
    \begin{gather}
        \| W \| = \bigl(\density^{\vdim}(\| V \|, \cdot) \circ \psi^{-1}\bigr)\, \HM^{\vdim} \restrict \psi \lIm \Sigma \rIm
        \\
        \text{and} \quad
        \density^{\vdim}(\| V \|, \psi^{-1}(x))
        = \density^{\vdim}(\| W \|, x)
        \quad \textrm{for  $ \HM^{\vdim} $ a.e.\ $ x \in \psi \lIm \Sigma \rIm $} \,.
    \end{gather}
    Moreover, recalling~\cite[Lemma B.2]{Santilli20a} we infer that
    \begin{equation}
        \label{push-forward basic eq}
        \Der \psi(b)\bigl\lIm \Tan^{\vdim}(\| V \|,b) \bigr\rIm
        =  \Tan^{\vdim}(\| W \|, \psi(b))
        \quad \textrm{for every $ b \in U $} \,.
    \end{equation}

    Finally, if $\alpha \in \ccspace{\grass{m}{\vdim}} $ then, employing again the area formula, we
    obtain
    \begin{multline}
        \textint{\cball{a}{r} \times \grass{m}{\vdim}}{}\alpha(S)  \ud W(x, S)
        \\
        = \textint{\psi^{-1} \lIm \cball{a}{r} \rIm}{}
        \alpha\bigl(\uD \psi(y) \lIm \Tan^{\vdim}(\|V \|, y) \rIm \bigr) \,
        \bigl| \tbwedge_{\vdim} \uD \psi(y) \circ \project{\Tan^{\vdim}(\| V \|, y)} \bigr|
        \ud \| V \|(y)
        \\
        = \textint{\cball{a}{r}}{}
        \alpha\bigl(\uD \psi(\psi^{-1}(x)) \lIm \Tan^{\vdim}(\|V \|, \psi^{-1}(x)) \rIm \bigr)\,
        \density^{\vdim}(\| V \|, \psi^{-1}(x))
        \ud \HM^{\vdim}(x)
        \\
        = \textint{\cball{a}{r}}{}
        \alpha\bigl( \Tan^{\vdim}(\|W\|, x) \bigr)
        \ud \| W \|(x)
    \end{multline}
    and it follows from \cite[2.9.8]{Federer1969} that
    \begin{displaymath}
        W^{(a)}(\alpha) = \lim_{r \to 0^{+}}
        \frac{\textint{\cball{a}{r} \times \grass{m}{\vdim}}{}\alpha(S) \ud W(x, S)}
        {\measureball{\| W \|}{\cball{a}{r}}}
        = \alpha(\Tan^{\vdim}(\|W\|, a))
    \end{displaymath}
    for $ \HM^{\vdim} $ a.e.\ $ a \in \psi \lIm \Sigma\rIm $.  
    
    By \cite[3.5(1)]{Allard1972} we infer that $ \psi_\# V \in \RVar{\vdim}(\Omega)$; if
    $ V \in \IVar{\vdim}(U) $, then $ \psi_\# V \in \IVar{\vdim}(\Omega) $.
\end{miniremark}

For the next  lemma recall Definition \ref{def:wmcv}.

\begin{lemma}
    \label{lem:image_mcv}
    Suppose \ref{push-forward basic}(1)-(5) and $ \| \delta_F V \| $ is a Radon measure.

    Then $\| \delta_{G} W \|$ is Radon and for each open set $ A \subseteq \R^m $ whose closure is
    a~compact subset of $ \Omega $ there holds
    \begin{equation}
        \label{lem:image_mcv eq1}
        M(A)^{-1} \psi_{\#} \| \delta_{F}V \|( A )
        \le \| \delta_{G} W \|(A)
        \le M(A) \psi_{\#} \| \delta_{F}V \|( A ) \,,
    \end{equation}
    where $M(A) = \sup \{ \| \uD \varphi(x) \| + \| (\uD \varphi(x))^{-1} \| : x \in A \} <
    \infty$.  Moreover, 
    \begin{equation}
        \label{lem:image_mcv eq2}
        \wmcv{G}(W,x) = \bigl| \tbwedge_{\vdim} \uD \varphi(x) \circ \project{\Tan^{\vdim}(\|W\|,x)} \bigr|\,
        \uD \varphi(x)^{*} \wmcv{F}(V,\varphi(x))
        \quad \textrm{for $\|W\|$ almost all $x \in \Omega $} \,.
    \end{equation}
\end{lemma}

\begin{proof}
    Suppose $ A \subseteq \R^m $ is open and $ \Clos{A} \subseteq \Omega $ is compact,
    $g \in \VF(\Omega)$ satisfies $\spt g \subseteq A$ and $\sup \im |g| \le 1$. Then, by
    \ref{rem:integral_representation_of_first_variation} and
    \ref{rem:perp:first_variation_pull_back} we see that
    \begin{multline}
        |\delta_{G} W(g)|
        = |\delta_{F}V(\langle g \circ \psi ,\, \uD \varphi \circ \psi \rangle)|
        \\
        = \bigl| \textint{}{} \uD \varphi(\psi(y)) (g \circ \psi(y))
        \bullet \boldsymbol{\eta}_F (V,y) \ud \| \delta_{F} V \|(y) \bigr|
        \\
        \le M(A) \textint{}{} |g(x)| \ud \psi_{\#}\| \delta_{F} V \|(x)
        \le M(A) \psi_{\#}\| \delta_{F} V \|(A) \,.
    \end{multline}
    Taking supremum over all $g \in \VF(\Omega)$ with $\spt g \subseteq A$ and $|g| \le 1$ we get
    \begin{equation}
        \label{eq:pull_back:first_var:deltaW_deltaV}
        \| \delta_{G} W \|(A)
        \le M(A) \psi_{\#} \| \delta_{F}V \|( A ) \,.
    \end{equation}
    Replacing $ G $, $ W $, $ A $ and $ \psi $ with $ F $, $ V $, $ \varphi[A] $ and $ \varphi $, we also obtain 
    \begin{equation}
        \label{eq:pull_back:first_var:deltaV_deltaW}
        \| \delta_{F} V \|(\varphi[A])
        \le M(A) \varphi_{\#} \| \delta_{G}W \|( \varphi[A] ), \quad  \psi_{\#} \| \delta_{F}V \|( A )
        \le M(A) \| \delta_{G} W \|(A)
        \,,
    \end{equation}
    whence we conclude that \eqref{lem:image_mcv eq1} holds and that $ \| \delta_G W \| $ is a Radon
    measure over $ \Omega $.

    Employing \ref{rem:integral_representation_of_first_variation}, \ref{def:wmcv}, and
    \ref{rem:perp:first_variation_pull_back} we obtain for $ g \in \VF(\Omega) $
    \begin{multline}
        \delta_{G} W(g)
        = \delta_{F} V(\langle g \circ \psi ,\, \uD \varphi \circ \psi \rangle)
        \\
        =
        \textint{}{} \boldsymbol{\eta}_{F}(V,y) \bullet \uD \varphi(\psi(y)) g( \psi(y) )
        \ud \| \delta_{F} V \|_{\mathrm{sing}}(y)
        \\
        - \textint{}{} \wmcv{F}(V,y) \bullet \uD \varphi(\psi(y)) g( \psi(y) ) \ud \| V \|(y)
        \,
    \end{multline}
    and, since $V = \varphi_{\#}W$, we conclude
    \begin{multline}
        \label{eq:pull_back:first_variation_of_W_wrt_G}
        \delta_{G} W(g) =
        \textint{}{} \uD \varphi(x)^* \boldsymbol{\eta}_{F}(V,\varphi(x)) \bullet  g(x)
        \ud \psi_{\#} \| \delta_{F} V \|_{\mathrm{sing}}(x)
        \\
        - \textint{}{} \bigl| \tbwedge_{\vdim} \uD \varphi(x) \circ \project{\Tan^{\vdim}(\|W\|,x)} \bigr| \,
        \uD \varphi(x)^* \wmcv{F}(V,\varphi(x)) \bullet g( x ) \ud \| W \|(x) \,.
    \end{multline}
     Define Borel sets
      \begin{displaymath}
          B = U \cap \bigl\{x : \mathbf{D} \bigl( \| \delta_F V \|, \| V \|, x \bigr) < \infty \bigr\}
          \quad \text{and} \quad
          C = \psi \lIm B \rIm \,.
      \end{displaymath}
      Using~\eqref{lem:image_mcv eq1} we get
      \begin{displaymath}
          C = \Omega \cap \bigl\{ y : \mathbf{D} \bigl( \| \delta_G W \|, \| W \|, y \bigr) < \infty \bigr\} \,.
      \end{displaymath}
      Recall~\ref{rem:integral_representation_of_first_variation} to see
      \begin{gather}
        \bigl( \| \delta_G W \|_{{\rm sing}} + \psi_\# \| \delta_F V \|_{{\rm sing}}\bigr)(C) = 0 \quad \text{and} \quad 
          \| W \|(\Omega \without C) = 0 \,.
    \end{gather}
    Define
    \begin{displaymath}
        Q(x) = \bigl| \tbwedge_{\vdim} \uD \varphi(x) \circ \project{\Tan^{\vdim}(\|W\|,x)} \bigr| \,
        \uD \varphi(x)^* \wmcv{F}(V,\varphi(x))
        \quad \textrm{for $ \| W \| $ a.e.\ $ x \in \Omega $} \,.
    \end{displaymath}
    If \eqref{lem:image_mcv eq2} were not true, we could find a compact set $ K \subseteq C $ such
    that $ \| W \|(K) > 0 $, $ u \in \sphere{\adim-1}$, and $ \gamma > 0 $ such that
    \begin{displaymath}
        \bigl( Q(x) - \wmcv{G}(W,x)\bigr) \bullet u \geq \gamma \quad \textrm{for $ x \in K $} \,.
    \end{displaymath}
    Then, we would choose an open set $ A \subseteq \Omega $ such that
    \begin{multline}
        \textint{A \without K}{} \bigl| Q(x) - \wmcv{G}(W,x) \bigr| \ud \| W \|(x)  \\ +
        \textint{A}{} \bigl|\uD \varphi(x)^* \boldsymbol{\eta}_{F}(V,\varphi(x))\bigr|
        \ud \psi_\#\| \delta_{F} V \|_{\mathrm{sing}}(x) + \| \delta_{G} W \|_{\mathrm{sing}}(A)
        \leq \tfrac{\gamma}{2} \| W \|(K)
    \end{multline}
    and a smooth function $ \zeta $ with compact support in $ A $ such that $ \zeta(x) = 1 $ for
    $ x \in K $ and $ 0 \leq \zeta(x) \leq 1 $ for $ x \in \Omega $.
    Using~\eqref{eq:pull_back:first_variation_of_W_wrt_G} we get
    \begin{multline}
        \textint{}{} \bigl( Q(x)- \wmcv{G}(W, x)\bigr) \bullet g( x ) \ud \| W \|(x) \\
        = \textint{}{} \uD \varphi(x)^* \boldsymbol{\eta}_{F}(V,\varphi(x)) \bullet g(x)
        \ud \psi_\#\| \delta_{F} V \|_{\mathrm{sing}}(x) - \textint{}{} \boldsymbol{\eta}_{G}(W,x) \bullet g(x)
        \ud \| \delta_{G} W \|_{\mathrm{sing}}(x)
    \end{multline}
    whenever $ g \in \VF(\Omega) $ so choosing $ g = \zeta \cdot u $ we would infer that
    \begin{displaymath}
        \textint{K}{} \bigl( Q(x)- \wmcv{G}(W, x)\bigr) \bullet u \ud \| W \|(x)
        \leq \tfrac{\gamma}{2} \| W \|(K) \,,
    \end{displaymath}
    a contradiction.
\end{proof}

\begin{lemma}
    \label{lem:pull_back:Lebesgue_points}
    If $ U $, $ V $ and $ F $ are as in \ref{push-forward basic}, $\| \delta_F V \|$ is a Radon
    measure and $ h $ is a mean $ F $-curvature vector of $ V $ (cf.~\ref{def:wmcv}), then there
    exists a set $B \subseteq U$ with $\|V\|(U \without B) = 0$ such that the following property
    holds.
    
    If $ \Omega \subseteq \R^{m} $, $ D $ is the set of all diffeomorphisms mapping $ U $ onto
    $ \Omega $ and $(b,\psi) \in B \times D $ then, defining $ a = \psi(b) $,
    $ \varphi = \psi^{-1} $, $ G = \varphi^\# F $, $ W = \psi_\# V $,
    \begin{multline}
        \label{lem:pull_back:Lebesgue_points eq}
        \tau(x) = \project{\Tan^{\vdim}(\|W\|,x)}
        \quad \text{and} \quad
        g(x) = \bigl| \tbwedge_{\vdim} \uD \varphi(x) \circ \tau(x) \bigr| \uD \varphi(x)^{*}  h( \varphi(x) )
        \\
        \text{for every $ x \in \Omega $ such that $ \Tan^{\vdim}(\| W \|,x) \in \grass{m}{\vdim} $,}
    \end{multline}
    there holds
    \begin{gather}
        \mathbf{D}\bigl(\|\delta_{G}W\|_{\mathrm{sing}}, \| W \|,a) = 0 \,,
        \quad
        \density^{\vdim}(\|W\|,a) = \density^{\vdim}(\|V\|,b) \,,
        \\
        \Tan^{\vdim}(\|W\|,a) =  \uD \psi(b) \lIm \Tan^{\vdim}(\|V\|,b) \rIm  \in \grass{m}{\vdim}  \,,
        \\
        \text{$a$ is a~$\|W \|$~Lebesgue point of $g$ and $\tau$} \,. 
    \end{gather}
\end{lemma}

\begin{proof}
    Suppose $ \Sigma = U \cap \{y : 0< \density^{\vdim}(\| V \|,y) < \infty\} $ and $B$ is the set of
    points $b \in \Sigma$ such that
    \begin{enumerate}[label=(\alph*)]
    \item there exists a $ {\vdim} $-dimensional submanifold $ Z_b \subseteq U $ of class $ \cnt{1} $ such
        that $ b \in Z_b $ and
        \begin{displaymath}
            \density^{\vdim}(\HM^{\vdim} \restrict(Z_b \without \Sigma), b) = \density^{\vdim}(\HM^{\vdim} \restrict(\Sigma \without Z_b), b) =0 \,,
        \end{displaymath}
    \item  $\Tan^{\vdim}(\|V\|,b) \in \grass{\adim}{\vdim} $, 
    \item $b$ is a~$\|V\|$ Lebesgue point of $h $ and $\project{\Tan^{\vdim}(\|V\|,\cdot)}$,
    \item $b$ is a~$\HM^{\vdim} \restrict \Sigma$ Lebesgue point of $\density^{\vdim}(\|V\|,\cdot)$, 
    \item $\mathbf{D}\bigl(\|\delta_{F}V\|_{\mathrm{sing}}, \| V \|,b) = 0 $.
    \end{enumerate}
    Then $\|V\|(U \without B) = 0$ by \cite[2.10.19(4), 2.9.9, 2.9.10]{Federer1969}
    and~\cite[3.5]{Allard1972}. We fix $ (b, \psi) \in B \times D $ and define
    \begin{gather}
        r_0 = \tfrac{1}{2}\dist(b, \R^m \without U)\,,
        \quad K = \cball{b}{r_0}\,,
        \\
        \varphi = \psi^{-1} \,,
        \quad
        a = \psi(b) \,,
        \quad
        G = \varphi^{\#}F \,, \quad W = \psi_\# V \,,
        \\
        \lambda = \sup \| \uD \psi[K] \| + \sup \| \uD \varphi[\psi(K)] \| +
        \Lip (\uD \psi|K)
        + \Lip \bigl(\uD \varphi|\psi \lIm K \rIm \bigr) \,,
        \\
        \text{$ \tau $ and $ g $ as in \eqref{lem:pull_back:Lebesgue_points eq}} \,,
        \\
        J_V\psi(y) =  \bigl| \tbwedge_{\vdim} \uD \psi(y) \circ \project{\Tan^{\vdim}(\|V\|,y)} \bigr|
        \quad \textrm{for $ \| V \| $ a.e.\ $ y \in U $} \,,
        \\
        J_W\varphi(x) =  \bigl| \tbwedge_{\vdim} \uD \varphi(x) \circ \tau(x) \bigr|
        \quad \textrm{for $ \| W \| $ a.e.\ $ x \in \Omega $} \,.
    \end{gather}
    Note that
    \begin{equation}
        \label{eq:perp:ball_inclusions_under_diffeomorphism}
        \begin{gathered}
            \cball{a}{\lambda^{-1} r} \subseteq \psi \lIm \cball{b}{r} \rIm \subseteq \cball{a}{\lambda r}
            \quad \text{for $0 < r \le r_0$}
            \\
            \text{and} \quad
            \cball{b}{\lambda^{-1} r} \subseteq \varphi \lIm \cball{a}{r} \rIm \subseteq \cball{b}{\lambda r}
            \quad \text{for $0 < r \le \lambda^{-1} r_0$} \,.
        \end{gathered}
    \end{equation}
    We infer 
    \begin{displaymath}
        \density^{\vdim}\bigl(\HM^{\vdim} \restrict\bigl(\psi(Z_b) \without \psi \lIm \Sigma \rIm\bigr), a\bigr)
        = \density^{\vdim}\bigl(\HM^{\vdim} \restrict\bigl(\psi \lIm \Sigma \rIm \without \psi(Z_b)\bigr), a\bigr) = 0 \,.
    \end{displaymath}
    Since $ \psi(Z_b) $ is a $ {\vdim} $-dimensional submanifold of class $ \cnt{1} $ and $ \density^{\vdim}(\HM^{\vdim} \restrict \psi(Z_b), a) = 1 $,  we see 
    \begin{equation} \label{eq:perp:density 1}
        \density^{\vdim}(\HM^{\vdim} \restrict \psi \lIm \Sigma \rIm, a) = 1 \,.
    \end{equation} 
    Noting that 
    \begin{multline}
        \inf \bigl \{ | \tbwedge_{\vdim} \uD \psi(y) \circ T |
        : y \in K ,\, T \in \pgrass{\adim}{\vdim} \bigr\}^{-1}
        \\
        = \sup \bigl \{ | \tbwedge_{\vdim} \uD \varphi(x) \circ T |
        : x \in \psi \lIm K \rIm ,\, T \in \pgrass{\adim}{\vdim} \bigr\}
        \le \lambda^{\vdim} \,,
    \end{multline}
    we conclude
    \begin{equation}
        \label{eq:perp: estimates jacobian}
        \lambda^{\vdim} \geq J_V \psi(y) \ge \lambda^{-\vdim}
        \quad \textrm{for $\|V\|$~almost all~$y \in K$} 
    \end{equation}
    and we use the area formula and~\eqref{eq:perp:ball_inclusions_under_diffeomorphism} to~see that
    \begin{equation}\label{eq:perp:density bounds 1}
        \lambda^{\vdim}\, \measureball{\| V \|}{\cball{b}{\lambda r}}
        \geq  \measureball{\| W \|}{\cball{a}{r}}
        = \textint{\varphi \lIm \cball{a}{r} \rIm}{}  J_V\psi(y) \ud \| V \|(y) 
        \geq  \lambda^{-\vdim}\, \measureball{\| V \|}{\cball{b}{\lambda^{-1}\, r}}
    \end{equation}
    for $ 0 < r < \lambda^{-1}\, r_0 $. Employing again the area formula~\cite[3.2.20]{Federer1969}
    we~observe that
    \begin{multline}
        \measureball{\| W \|}{\cball{a}{r}}
        = \textint{\varphi \lIm \cball ar \rIm \cap \Sigma}{}
        J_V\psi(y)\,\density^{\vdim}(\|V\|,y) \ud \HM^{\vdim}(y)
        \\
        = \textint{\varphi \lIm \cball ar \rIm \cap \Sigma}{}
        \bigl( \density^{\vdim}(\|V\|,y) - \density^{\vdim}(\|V\|,b) \bigr)
        J_V\psi(y)
        \ud \HM^{\vdim}(y)
        \\
        \quad  +  \density^{\vdim}(\|V\|,b) \,
        \textint{\varphi \lIm \cball ar \rIm \cap \Sigma}{}
        J_V\psi(y)
        \ud \HM^{\vdim}(y)
        \\
        = \textint{\varphi \lIm \cball ar \rIm \cap \Sigma}{}
        \bigl( \density^{\vdim}(\|V\|,y) - \density^{\vdim}(\|V\|,b) \bigr)
        J_V\psi(y)
        \ud \HM^{\vdim}(y)
        \\
        \quad  +  \density^{\vdim}(\|V\|,b) \,
        \HM^{\vdim}\bigl(\cball{a}{r} \cap \psi \lIm \Sigma \rIm \bigr)
    \end{multline}
    and, since $b$ is a~$\HM^{\vdim} \restrict \Sigma$ Lebesgue point of
    $\density^{\vdim}(\|V\|,\cdot)$ and $ \density^{\vdim}(\HM^{\vdim} \restrict \Sigma, b) = 1 $, we conclude
    employing \eqref{eq:perp:ball_inclusions_under_diffeomorphism}, \eqref{eq:perp:density 1} and
    \eqref{eq:perp: estimates jacobian} that
    \begin{displaymath}
        \density^{\vdim}(\| W \|, a) = \density^{\vdim}(\| V \|, b) \,.
    \end{displaymath}
    Moreover, since $ \mathbf{D}\bigl(\| \delta_F V \|_{\mathrm{sing}}, \| V \|,b\bigr) = 0 $ and
    $ 0 <\density^{\vdim}(\| V \|, b) < \infty $, we employ~\ref{lem:image_mcv}\eqref{lem:image_mcv eq1}
    and \eqref{eq:perp:density bounds 1} to see that
    \begin{displaymath}
        \mathbf{D}\bigl( \|\delta_{G}W\|_{\mathrm{sing}}, \| W \|, a \bigr) = 0 \,.
    \end{displaymath}
    Recalling \ref{push-forward basic}\eqref{push-forward basic eq}, we see that
    \begin{displaymath}
        \Tan^{\vdim}(\| W \|, a) = \Der \psi(b)[\Tan^{\vdim}(\| V \|, b)] \in \grass{m}{\vdim}
    \end{displaymath}
    and we estimate for $0 < r \le \lambda^{-1} r_0$
    \begin{multline}
        \label{eq:perp:tangent_Lebesgue_points}
        \textint{\cball ar}{} 
        | \tau(x) - \tau(a) | 
        \ud \|W\|(x)
        = \textint{\varphi \lIm \cball ar \rIm}{}
        | \tau(\psi(y))- \tau(\psi(b)) |
        J_V \psi(y)
        \ud \|V\|(y)
        \\
        = \textint{\varphi \lIm \cball ar \rIm}{}
        \bigl| \project{\bigl( \Der \psi(y)[\Tan^{\vdim}(\|V\|,y)]\bigr)}
        - \project{\bigl( \Der \psi(y)[\Tan^{\vdim}(\|V\|,b)]\bigr)} \bigr|
        J_V \psi(y)
        \ud \|V\|(y)
        \\
        \le \lambda^{\vdim} 
        \textint{\cball{b}{\lambda r}}{}
        \bigl| \project{\bigl( \Der \psi(y)[\Tan^{\vdim}(\|V\|,y)]\bigr)}
        - \project{\bigl( \Der \psi(y)[\Tan^{\vdim}(\|V\|,b)]\bigr)} \bigr|
        \ud \|V\|(y) \,.
    \end{multline}
    
    Applying~\ref{mr:Lebesgue_points_and_taking_image} with $\uD \psi$ in place of $h$, we see that
    $ b $ is a $ \| V \| $ Lebesgue point of
    $ \bigl[ U \ni y \mapsto \project{\bigl( \Der \psi(y)[\Tan^{\vdim}(\|V\|,y)]\bigr)} \bigr] $,
    and we~conclude, employing \eqref{eq:perp:density bounds 1}, that
    \begin{displaymath}
        \lim_{r \to 0} \bigl( \measureball{\| W \|}{\cball ar} \bigr)^{-1} \textint{\cball ar}{}
        | \tau(x) - \tau(a) |
        \ud \|W\|(x) = 0 \,.
    \end{displaymath}

    Noting that $ \tau(\psi(y)) = \project{\bigl(\Der \psi(y)[\Tan^{\vdim}(\| V \|,y)] \bigr)} $ for
    $ \| V \| $ almost all $ y \in U $, we invoke Lemma~\ref{mr:jacobian_multiplicative} with
    $ \Der \varphi(\psi(y)) $ and $ \Der \psi(y) \circ \project{\Tan^{\vdim}(\| V \|,y)} $ in place of
    $ M $ and $ L $, to conclude
    \begin{displaymath}
        J_W\varphi(\psi(y)) J_V \psi(y) = \bigl| \tbwedge_{\vdim} \project{\Tan^{\vdim}(\| V \|,y)}\bigr| = 1
        \quad \text{for $\|V\|$ almost all $y \in U$} \,. 
    \end{displaymath}
    Moreover,  since $\Lip \uD \varphi|\psi \lIm K \rIm \le \lambda < \infty$, we have
    \begin{displaymath}
        \kappa = \Lip \bigl[
        \psi \lIm K \rIm \times \pgrass{\adim}{\vdim} \ni (x,T) \mapsto \bigl| \tbwedge_{\vdim} \uD \varphi(x) \circ T \bigr|
        \bigr] < \infty \,.
    \end{displaymath}
    Henceforth, we get  for $0 < r \le \lambda^{-1} r_0$  
    \begin{flalign*}
      & \textint{\cball ar}{}
        | g(x) - g(a) |
        \ud \|W\| (x)
      \\
      & \quad = 
        \textint{\varphi \lIm \cball ar \rIm}{}
        | g(\psi(y)) - g(\psi(b))|\, J_V \psi(y)
        \ud \|V\| (y)
      \\
      & \quad \le \textint{\cball{b}{\lambda r}}{}
        \bigl| \uD \varphi(\psi(y))^* (h(y) - h(b)) \bigr|
        \ud \| V \|(y)
      \\
      & \quad \quad + \textint{\cball{b}{\lambda r}}{} \bigl|
        \bigl( \uD \varphi(\psi(y))^*- \uD \varphi(\psi(b))^* \bigr )h(b)
        \bigr|  \ud \| V \|(y)
      \\
      & \quad \quad + \textint{\cball{b}{\lambda r}}{} \bigl| \uD \varphi(a)^* h(b) \bigr|\,
        \bigl| J_W\varphi(\psi(y)) - J_W \varphi(\psi(b)) \bigr) \bigr| \, J_V \psi(y)\,
        \ud \|V\| (y)
      \\
      & \quad \le 
        \lambda\, \textint{\cball{b}{\lambda r}}{}
        \bigl| h(y) - h(b) \bigr|
        \ud \| V \|(y)
      \\
      & \quad \quad 
        + \lambda^3 r\, \bigl| h(b) \bigr|\, \measureball{\| V \|}{\cball{b}{\lambda r}}
      \\
      &\quad \quad + \lambda^{\vdim+1}
        | h(b) | \kappa \textint{\cball{b}{\lambda r}}{}
        (\lambda^2 r + \| \tau(\psi(y)) - \tau(\psi(b))\|)
        \ud \|V\| (y) \,. 
    \end{flalign*}
    Dividing by $\measureball{\| W \|}{\cball ar} $ and letting $ r \to 0 $, we employ
    \eqref{eq:perp:density bounds 1} to conclude that
    \begin{displaymath}
        \lim_{r \to 0} \textfint{\cball ar}{} | g(x) - g(a) | \ud \|W\| (x) = 0 \,.
        \qedhere
    \end{displaymath}
\end{proof}

\begin{remark}
    The function $ g $ in \ref{lem:pull_back:Lebesgue_points} is a mean $ G $-curvature of $ W $ by
    \ref{lem:image_mcv}\eqref{lem:image_mcv eq2}.
\end{remark}

\begin{lemma}
    \label{lem:perp:pull_back_by_auto}
    Suppose $\Omega \subseteq \R^m$ is open, $ F $ is a $ {\vdim} $-dimensional $ \cnt{1} $-integrand on
    $ \Omega $, $A \in \End{\R^m}$ is an isomorphism and $T \in \grass mk$.
    
    Then the following statements hold.
    \begin{enumerate}
    \item
        \label{lem:perp:pull_back_by_auto_1}
        If $A \circ \project{T} = \project{T}$ then $  B_{A^{\#}F}(x,T) = A^* \circ B_{F}(Ax,T) $ for $ x \in \Omega $.
    \item
        \label{lem:perp:pull_back_by_auto_2}
        If $A \circ \project{T} = \project{T}$ and $ \im (A \circ \perpproject{T}) = \ker P_F(Ax,T)$
        then
  	\begin{displaymath}
            P_{A^{\#}F}(x,T) = \project{T}
            \quad \text{and} \quad
            \langle \sigma^*X ,\, \uD A^{\#}F (x,T) \rangle  = 0
            \quad \text{for $X \in \Tan(\grass{\adim}{\vdim},T)$} \,.
  	\end{displaymath}
  	In~particular, there exists $0 < \Gamma < \infty$ such that
  	\begin{displaymath}
            |A^{\#}F(x,S) - A^{\#}F(x,T)| \le \Gamma \| \project{S} - \project{T} \|^2
            \quad \text{for $S \in \grass mk$}\,.
  	\end{displaymath}
    \end{enumerate}
\end{lemma}

\begin{proof}
    We prove \ref{lem:perp:pull_back_by_auto_1}. Let $L \in \End{\R^{\adim}}$ and set
    $L_{t} = \id{\R^{\adim}} + t L$ for $t \in \R$. Recall~\ref{mr:jacobian_multiplicative} and note
    that since $A \circ \project{T} = \project{T}$ we~get
    \begin{gather}
        A \circ L_{t} \circ \project{T}
        = A \circ (\id{\R^{\adim}} + t L) \circ \project{T}
        = \bigl( \id{\R^{\adim}} + t (A \circ L) \bigr) \circ \project{T} \,,
        \\
        A \bigl\lIm L_t \lIm T \rIm \bigr\rIm  = \im (A \circ L_{t} \circ \project{T}) \,,
        \\
        (A \circ \eqproject{L_t \lIm T \rIm} ) \circ (L_{t} \circ \project{T})
        = A \circ L_{t} \circ \project{T} \,,
        \\
        |\tbwedge_{\vdim} (A \circ \eqproject{L_t \lIm T \rIm})|
        \cdot |\tbwedge_{\vdim} (L_{t} \circ \project{T})|
        = |\tbwedge_{\vdim} (A \circ L_{t} \circ \project{T})| \,; 
    \end{gather}
    hence,
    \begin{multline}
        A^{\#}F(x, L_{t} \lIm T \rIm) \, |\tbwedge_{\vdim} (L_{t} \circ \project{T})|
        = F \bigl( Ax, \im A \circ L_{t} \circ \project{T} \bigr) \, |\tbwedge_{\vdim} (A \circ L_{t} \circ \project{T})|
        \\
        = F\bigl( Ax, \im \bigl( (\id{\R^{\adim}} + t (A \circ L)) \circ \project{T}\bigr) \bigr) \,
          |\tbwedge_{\vdim} (\id{\R^{\adim}} + t (A \circ L)) \circ \project{T}| \,.
    \end{multline}
    Combining this with~\ref{rem:BF_as_derivative} yields
    \begin{multline}
        B_{A^{\#}F}(x,T) \bullet L
        = \left. \tfrac{\ud}{\ud t} \right|_{t=0}
        A^{\#}F(x, L_t[T])\, |\tbwedge_{\vdim} L_{t} \circ \project{T}|
        \\
        = \left. \tfrac{\ud}{\ud t} \right|_{t=0}
        F\bigl( Ax, \im \bigl( (\id{\R^{\adim}} + t (A \circ L)) \circ \project{T} \bigr) \bigr) \,
        |\tbwedge_{\vdim} (\id{\R^{\adim}} + t (A \circ L)) \circ \project{T}|
        \\
        = B_F(Ax,T) \bullet (A \circ L)
        = A^* \circ B_F(Ax,T) \bullet L
        \,.
    \end{multline}

    We prove \ref{lem:perp:pull_back_by_auto_2}.  Since
    $ \im (A \circ \perpproject{T}) = \ker P_F(Ax,T)$, we deduce from
    \ref{lem:perp:pull_back_by_auto_1} that
    \begin{displaymath}
	B_{A^{\#}F}(x,T)^* \circ \perpproject{T}
	= B_{F}(Ax,T)^* \circ A \circ \perpproject{T} = 0 \,,
        \quad T^\perp \subseteq \ker P_{A^{\#}F}(x,T)\,.
    \end{displaymath}
    On the other hand, by \ref{rem:P_projection}, we see that $ T = \im P_{A^{\#}F}(x,T) $ and
    $ T \cap \ker P_{A^{\#}F}(x,T) = \{0\} $ and we conclude that
    \begin{displaymath}
        \ker P_{A^{\#}F}(x,T) = T^\perp
        \quad \text{and} \quad P_{A^{\#}F}(x,T) = \project{T} \,.
    \end{displaymath}
    The formulas for $B_{A^{\#}F}(T)$ and $P_{A^{\#}F}(T)$ given in~\ref{rem:BF_as_derivative}
    and~\ref{def:related_projections} yield immediately that
    \begin{displaymath}
	\langle \sigma^{*} X ,\, \uD A^{\#}F (x,T) \rangle = 0 \,,
    \end{displaymath}
    whenever $X \in \Tan(\grass{\adim}{\vdim},T)$. The postscript now follows by the Taylor
    formula~\cite[3.1.11 p.~220]{Federer1969} combined with compactness of the
    Grassmannian~$\grass{\adim}{\vdim}$.
\end{proof}

\section{Perpendicularity and locality of the mean curvature}
\label{sec:locality}

\begin{miniremark}
    \label{mr:def_varphi_T}
    Suppose $F$ is a $\vdim$-dimensional autonomous integrand of class~$\cnt{1}$ on $ \R^m $.
    Recall~\ref{rem:P_projection} to define for $ T \in \grass{m}{\vdim} $,
    \begin{displaymath}
        \varphi_T \in \End{\R^{\adim}}
        \quad \text{by} \quad
        \varphi_T = \project{T} + Q_F(T) = \project{T} + \id{\R^{\adim}} - P_{F}(T)\,.
    \end{displaymath}  
    We notice that $ \varphi_T $ is invertible with $ \varphi_T^{-1} = \perpproject{T} + P_F(T) $.
\end{miniremark}

\begin{theorem}
    \label{thm:mcv-orto}
    Suppose $F$ is a $\vdim$-dimensional autonomous integrand of class~$\cnt{1}$ on $ \R^m $,
    $ U \subseteq \R^{\adim} $ is open, $ V \in \IVar{\vdim}(U) $ such that $ \| \delta_{F} V \| $
    is a~Radon measure over $ U $ and
    \begin{gather}
        \Sigma = U \cap \bigl\{ y : 0< \density^{\vdim}(\| V \|,y) < \infty \bigr\}\,.
    \end{gather}
    For $V$ almost all $(b,T)$ and $0 < r < \dist(b,\R^{\adim} \without U)$, let $\varphi_T$ be as
    in~\ref{mr:def_varphi_T} and set
    \begin{displaymath}
        X(b,T,r) = T \cap \bigl\{
        y : \textsum{z \in \Sigma \cap \varphi_T[\cball br] \cap P_F(T)^{-1}\{ y \}}{}
        \density^{\vdim}(\|V\|,z) 
        = \density^{\vdim}(\|V\|,b)
        \bigr\} \,.
    \end{displaymath}
    Define $E$ to be the set of $(b,T) \in U \times \grass{m}{\vdim}$ for which
    \begin{gather}
        \density^{\vdim}(\|V\|,b) \in \natp \,,
        \quad
        T = \Tan^{\vdim}(\|V\|,b) \in \grass{\adim}{\vdim} \,,
        \\
        \label{eq:perp:height_decay}
        \lim_{r \to 0^{+}} r^{-\vdim-3}
        \textint{\cball br}{} | \perpproject{T}(y-b) |^2 \ud \|V\|(y) = 0 \,,
        \\
        \label{eq:perp:tilt_decay}
        \lim_{r \to 0^{+}} r^{-\vdim-1}
        \textint{\cball br \times \grass{\adim}{\vdim}}{} \| \project{S} - \project{T} \|^2 \ud V(y,S) = 0 \,,
        \\
        \label{eq:perp:bad_set_decay}
        \text{and} \quad
        \begin{multlined}
            \lim_{r \to 0^{+}} r^{-\vdim-1}  \HM^{\vdim} \bigl( T \cap \cball{P_F(T)b}{r} \without X(b,T,r) \bigr) =0
            \\
            \phantom{\hspace{7em}}
            \lim_{r \to 0^{+}} r^{-\vdim-1}  \|V\|\bigl( \varphi_T[\cball br] \without P_{F}(T)^{-1} \lIm X(b,T,r) \rIm \bigr)= 0
            \,.
        \end{multlined}
    \end{gather}
    Then
    \begin{displaymath}
        \project{T}(\mcv{F}(V,b)) = 0 
        \quad \text{for $V$ almost all $(b,T) \in E$} \,.
    \end{displaymath}
\end{theorem}

\begin{proof}
    Let $ h $ be a mean $ F $-curvature of $ V $ and let $B$ be the set given by
    Lemma~\ref{lem:pull_back:Lebesgue_points}. Fix $(b,T) \in E \cap (B \times \grass{m}{\vdim})$
    and set $ \mu = \density^{\vdim}(\|V\|,b) \in \natp $. Without loss of generality, assume
    $b = 0$ and set $ \varphi = \varphi_T $ and $ \psi = \varphi^{-1} $. Recalling
    Lemma~\ref{lem:pull_back:Lebesgue_points} define
    \begin{displaymath}
        \Omega = \psi \lIm U \rIm \,,
        \quad
        W = \psi_{\#}V \,,
        \quad
        G = \varphi^{\#}F \,,
    \end{displaymath}
    and also
    \begin{multline}
        \tau(x) =  \project{\Tan^{\vdim}(\|W\|,x)}
        \quad \text{and} \quad
        g(x) = \bigl| \tbwedge_{\vdim} \varphi \circ \tau(x) \bigr|  \varphi^{*} h(x)
        \\
        \text{whenever $ \Tan^{\vdim}(\| W \|,x) \in \grass{\adim}{\vdim} $} \,.
    \end{multline}
    Notice that $0 = \psi(b) $, $G$ is a $\vdim$-dimensional autonomous integrand of class
    $ \cnt{1} $ in~$\Omega$, $W \in \IVar{\vdim}(\Omega)$ and $\| \delta_{G} W \|$ is a~Radon
    measure.  Moreover, Lemma~\ref{lem:pull_back:Lebesgue_points} yields
    \begin{equation}
        \label{eq:perp:Lebesgue_points}
        \begin{gathered}
            \text{$0$ is a~$\|W\|$~Lebesgue point of $g$ and $\tau$} \,, \\
            \density^{\vdim}(\|\delta_{G} W\|_{\mathrm{sing}},0) = 0 \,,
            \quad
           \density^{\vdim}(\| W\|,0)  = \mu \,.
        \end{gathered}
    \end{equation}
    Employing~\ref{lem:perp:pull_back_by_auto} with $ \varphi $ in place of $ A $ and
    recalling~\ref{rem:P_projection}, one sees that
    \begin{gather}
        G(T)  = F(T)\,,
        \quad
        \tau(0) = \project{\psi[T]} = \project{T}\,,
        \quad
 	B_G(T) = \varphi^\ast \circ B_F(T)\,,
        \\
        \label{eq:perp:projections}
        \project{T} \circ \psi = P_F(T)\,,
        \quad
        P_F(T) \circ \varphi = P_G(T) = \project{T}\,,
        \quad
        Q_G(T) = \perpproject{T} \,,
        \\
        \label{eq:perp:fibers of projections}
        \varphi\bigl(\project{T}^{-1}\{y\}\bigr) = P_F(T)^{-1}\{y\}
        \quad
        \text{for $ y \in \R^m $,}
    \end{gather}
    and that there exists $0 < \Delta_1 < \infty$ such that
    \begin{equation}
	\label{eq:perp:G_quadratic}
	|G(S) - G(T)| \le \Delta_1 \| \project{S} - \project{T} \|^2
	\quad \text{for $S \in \grass mk$}\,.
    \end{equation}

    Let $L = \Lip \bigl[ \pgrass{\adim}{\vdim} \ni R \mapsto \project{\psi \lIm \im R \rIm} \bigr]$
    and note that $L < \infty$
    by~\ref{mr:taking_image_of_an_endomorphism}. Recall
    $\psi \lIm T \rIm = T$, $ \Der \psi(y) = \psi $ for $ y \in \R^m $ and \ref{push-forward basic}, to get
    \begin{multline}
        \textint{\cball 0r \times \grass{\adim}{\vdim}}{} \| \project{S} - \project{T} \|^2 \ud W(x,S)
        = \textint{\cball 0r}{} \| \tau(x) - \project{T} \|^2 \ud \|W\|(x)
        \\
        = \textint{\cball 0r}{}
        \| \project{ \psi \lIm \Tan^{\vdim}(\|V\|,\varphi(x)) \rIm } - \project{\psi \lIm T \rIm} \|^2
        \ud \|W\|(x)
        \\
        \le L^2 \textint{\cball 0r}{}
        \| \project{ \Tan^{\vdim}(\|V\|,\varphi(x))  } - \project{T} \|^2
        \ud \| \psi_{\#}V\|(x)
        \\
        = L^2 \textint{\varphi \lIm \cball 0r \rIm \times \grass{m}{\vdim}}{}
        \| \project{S} - \project{T} \|^2 | \tbwedge_{\vdim} \psi \circ S | \ud V(y,S)
        \\
        \le L^2 \| \psi \|^{\vdim} \textint{\varphi \lIm \cball 0r \rIm \times \grass{m}{\vdim}}{}
        \| \project{S} - \project{T} \|^2 \ud V(y,S) \,.
    \end{multline}
    Since $(0,T) \in E$ conclude from~\eqref{eq:perp:tilt_decay} that
    \begin{equation}
        \label{eq:perp:tilt_decay_for_W}
        \lim_{r \to 0^{+}} r^{-\vdim-1}
        \textint{\cball 0r \times \grass{\adim}{\vdim}}{}
        \| \project{S} - \project{T} \|^2
        \ud W(x,S) = 0 \,.
    \end{equation}
    Analogously, noting that $\perpproject{T} \circ \psi = \perpproject{T}$ and
    using~\eqref{eq:perp:height_decay}, derive
    \begin{gather}
        \label{eq:perp:height_decay_for_W}
        \lim_{r \to 0^{+}} r^{-\vdim-3}
        \textint{\cball 0r}{} |\perpproject{T}x|^2 \ud \|W\|(x) = 0 
        \,.
    \end{gather}
   
    Fix $u \in T \cap \sphere{\adim-1}$ and $ 0 < \lambda < 1 $. Then choose a smooth
    function
    \begin{displaymath}
        \zeta_\lambda : \R \to \R \cap \{ t : 0 \le t \le 1 \}
    \end{displaymath}
    such that $ \zeta_\lambda (t) = 1 $ for $t \le \lambda $ and
    $ \spt \zeta_\lambda \subseteq \R \cap \{ t : t < 1 \} $. For $0 < r < r_0$ define
    \begin{gather}
        \rchi_{\lambda, r}(x) = \zeta_\lambda(|x|/r) \quad \text{for $x \in \Omega$} 
    \end{gather}
    and observe that
    \begin{gather}
        \label{eq:perp:null_int}
        \uD \rchi_{\lambda, r}(-x) = -\uD \rchi_{\lambda, r}(x) \quad \text{for $x \in \R^m$} \,,\quad
        \textint{T}{} \grad \rchi_{\lambda, r} \ud \HM^{\vdim}=0 \,,
        \\
        \label{eq:perp:gr_derivative}
        \text{and} \quad
        \delta_GW(\rchi_{\lambda, r}u)
        = \textint{}{} \bigl(P_G(S)u \bullet \nabla \rchi_{\lambda, r}(x)\bigr)\, G(S)
        \ud W(x,S) \,. 
    \end{gather}
    Given $0 < r < r_0$ set $ Y_r = X(0,T,r) $ and define
    \begin{align}
      a_0(r) &= \textint{\cball 0r \times \grass{\adim}{\vdim}}{} \,
               P_G(S) u \bullet \grad \rchi_{\lambda, r}(x) \bigl( G(S) - G(T) \bigr)
               \ud W(x,S)  \,,
      \\
      a_1(r) &= \textint{\cball 0r \times \grass{\adim}{\vdim}}{} \,
               \bigl( P_G(S) - P_G(T) \bigr) u \bullet \grad \rchi_{\lambda, r}(x)
               \ud W(x,S) \,,
      \\
      a_2(r) &= \textint{\cball 0r \without \project{T}^{-1} \lIm Y_r \rIm}{}
               \grad \rchi_{\lambda, r}(x)
               \ud \|W\|(x) \,,
      \\
      a_3(r) &= \textint{\cball 0r \cap \project{T}^{-1} \lIm Y_r \rIm}{}
               \bigl( \grad \rchi_{\lambda, r}(x) - \grad \rchi_{\lambda, r}(\project{T}x) \bigr)
               \ud \|W\|(x) \,,
      \\
      a_4(r) &= \textint{(\cball 0r \cap \project{T}^{-1} \lIm Y_r \rIm) \times \grass{\adim}{\vdim}}{}
               \grad \rchi_{\lambda, r}(\project{T}x)
               \bigl( 1 - | \tbwedge_{\vdim} \project{T} \circ \project{S} | \bigr)
               \ud W(x,S) \,,
      \\
      a_5(r) &= -\mu \textint{T \cap \cball 0r \without Y_r}{}
               \grad \rchi_{\lambda, r}(y)
               \ud \HM^{\vdim}(y) \,.
    \end{align}
    by~\eqref{eq:perp:gr_derivative} it follows that
    \begin{multline}
        \delta_{G} W(\rchi_{\lambda, r}\, u) 
        = a_0(r) + G(T) \bigl( a_1(r) + P_G(T) u \bullet
        (a_2(r) + a_3(r) + a_4(r)) \bigr) \\
        + G(T)\, P_G(T)u \bullet  \textint{\cball 0r \cap \project{T}^{-1} \lIm Y_r \rIm}{}
        \grad \rchi_{\lambda, r}(\project{T}x)
        | \tbwedge_{\vdim} \project{T} \circ \project{\Tan^{\vdim}(\| W \|, x)} |
        \ud \| W \|(x);
    \end{multline}
    next, deduce using~\ref{push-forward basic}, coarea formula~\cite[3.2.22]{Federer1969},
    \eqref{eq:perp:fibers of projections}, and \eqref{eq:perp:null_int} that
    \begin{multline}
	\textint{\cball 0r \cap \project{T}^{-1} \lIm Y_r \rIm}{}
	\grad \rchi_{\lambda, r}(\project{T}x)
        | \tbwedge_{\vdim} \project{T} \circ \project{\Tan^{\vdim}(\| W \|, x)} |  \ud \| W \|(x)
        \\
	= \textint{\cball 0r \cap \project{T}^{-1} \lIm Y_r \rIm \cap \psi \lIm \Sigma \rIm}{}
	\grad \rchi_{\lambda, r}(\project{T}x)
        | \tbwedge_{\vdim} \project{T} \circ \project{\Tan^{\vdim}(\| W \|, x)} | \,
        \density^{\vdim}(\| V \|, \varphi(x) )
        \ud\HM^{\vdim}(x)
        \\
	= \textint{Y_r \cap \project{T}\lIm \cball 0r \cap \psi \lIm \Sigma \rIm  \rIm}{}
	\grad \rchi_{\lambda, r}(y)
        \textint{P_F(T)^{-1}\{y\}\cap \varphi \lIm \cball 0r \rIm \cap \Sigma}{}
        \density^{\vdim}(\| V \|, u)
        \ud \HM^0(u) \ud \HM^{\vdim}(y)
        \\
        = \textint{Y_r \cap  \cball 0r}{}
	\grad \rchi_{\lambda, r}(y)
        \textint{P_F(T)^{-1}\{y\}\cap \varphi \lIm \cball 0r \rIm \cap \Sigma}{}
        \density^{\vdim}(\| V \|, u) \ud \HM^0(u) \ud \HM^{\vdim}(y)
        \\
	= \mu \textint{Y_r \cap  \cball 0r}{} 
	\grad \rchi_{\lambda, r}(y) \ud \HM^{\vdim}(y) = a_5(r)
    \end{multline}
    and conclude that
    \begin{equation}
	\label{eq:perp:first_variation_of_W}
	\delta_{G} W(\rchi_{\lambda, r}\, u) 
        = a_0(r) + G(T) \bigl( a_1(r) + P_G(T) u \bullet
        (a_2(r) + a_3(r) + a_4(r) + a_5(r)) \bigr)\,.
    \end{equation}

    Let us estimate each term separately. Noting
    \begin{gather}
        u = \project{T} u = P_G(T) u \,,
	\quad
	(P_G(S) - P_G(T)) u
	=  - Q_G(S) \circ P_G(T) u
        \\
        \text{and}\quad  Q_G(S) \circ Q_G(S) = Q_G(S) \quad \text{for $ S \in \grass{m}{\vdim} $} \,,
    \end{gather}
    compute
    \begin{multline}
        (P_G(S) - P_G(T)) u \bullet \grad \rchi_{\lambda, r}(x)
        = -  Q_G(S) \circ P_G(T) u \bullet Q_G(S)^* \grad \chi_{\lambda, r}(x)
        \\
        = -Q_G(S) \circ P_G(T) u \bullet \bigl(Q_G(S)^* - Q_G(T)^*\bigr) \grad \rchi_{\lambda, r}(x)
        \\
        - Q_G(S) \circ P_G(T) u \bullet Q_G(T)^* \grad \rchi_{\lambda, r}(x)
        \\
        = (P_G(S) - P_G(T)) u \bullet \bigl(Q_G(S)^* - Q_G(T)^*\bigr) \grad \rchi_{\lambda, r}(x)
        \\
        +  (P_G(S) - P_G(T)) u \bullet \perpproject{T} \grad \rchi_{\lambda, r}(x)
        \quad \text{for $ S \in \grass{m}{\vdim} $} \,.
    \end{multline}
    Hence, choosing $ \Delta_2 > 0 $ so that 
    \begin{displaymath}
        \| P_G(S_1)- P_G(S_2) \| \leq \Delta_2 \, \| \project{S_1} - \project{S_2} \|
        \quad \text{and} \quad
        \| Q_G(S_1)^*- Q_G(S_2)^* \| \leq \Delta_2 \, \|  \project{S_1} - \project{S_2} \| 
    \end{displaymath}
    for $ S_1, S_2 \in \grass{m}{\vdim} $, and noting that
    \begin{equation}
        \label{eq:perp:grad_chi_estimate}
        \begin{gathered}
            \grad \rchi_{\lambda, r}(x)
            = r^{-1} \zeta_\lambda'(|x|/r) \tfrac{x}{|x|}
            \quad \text{for $ x \in \R^m $}
            \\
            \text{and} \quad
            \grad \rchi_{\lambda, r}(x) = 0
            \quad \text{for $ x \notin \cball 0r \without \cball{0}{\lambda r} $}
        \end{gathered}
    \end{equation}
    use H\"older inequality to find that
    \begin{multline}
        |a_1(r)|
        \leq r^{-1}\,\Delta_2^2 \,\sup| \zeta_\lambda'|\,
        \textint{\cball 0r \times \grass{\adim}{\vdim}}{}
        \| \project{S} - \project{T} \|^2
        \ud W(x,S)
        \\
        +  r^{-1}\,\Delta_2\, \sup| \zeta_\lambda'|\,
        \textint{\cball 0r \without \cball{0}{\lambda r}}{}
        \| \project{S} - \project{T} \|\, \bigl| \perpproject{T}\bigl(\tfrac{x}{|x|}\bigr)\bigr|
        \ud W(x,S)
        \\
        \leq r^{-1}\,\Delta_2^2 \,\sup| \zeta_\lambda'|\,
        \textint{\cball 0r \times \grass{\adim}{\vdim}}{}
        \| \project{S} - \project{T} \|^2
        \ud W(x,S)
        \\
        + \Delta_2 \sup| \zeta_\lambda'|
        \Bigl( r^{-1} \textint{\cball 0r \times \grass{\adim}{\vdim}}{}
        \| \project{S} - \project{T} \|^2
        \ud W(x,S)\Bigr)^{1/2}
        \Bigl( \lambda^{-2}r^{-3}\textint{\cball 0r}{}
        | \perpproject{T}(x)|^2
        \ud \| W \|(x)\Bigr)^{1/2} \,. 
    \end{multline}
    Combining~\eqref{eq:perp:G_quadratic} and \eqref{eq:perp:grad_chi_estimate} yields also
    \begin{displaymath}
  	\bigl| P_G(S) u \bullet \grad \rchi_{\lambda, r}(x) (G(S) - G(T)) \bigr|
  	\le r^{-1} \Lip \zeta_\lambda \sup \bigl( \im \|P_G \| \bigr)
        \Delta_1 \| \project{S} - \project{T} \|^2 \,,
    \end{displaymath}
    so summoning \eqref{eq:perp:tilt_decay_for_W} and \eqref{eq:perp:height_decay_for_W} gives
    \begin{equation} \label{eq:perp:A0A1_estimate}
  	\lim_{r\to 0} r^{-{\vdim}} (a_0(r) + a_1(r)) = 0 \,.
    \end{equation}
    
    Since $ \project{T} \circ \psi = P_F(T) $ one gets
    \begin{displaymath}
        a_2(r) = \textint{\varphi[\cball 0r] \without P_F(T)^{-1} \lIm Y_r \rIm}{}
        \grad \rchi_{\lambda, r}(\psi(y))\, | \tbwedge_{\vdim}\psi \circ \project{S}|
        \ud V(y,S) \,;
    \end{displaymath}
    hence, it follows by~\eqref{eq:perp:bad_set_decay} and~\eqref{eq:perp:grad_chi_estimate} that
    \begin{equation}
        \label{eq:perp:A2A5_estimate}
        \lim_{r \to 0^+} r^{-\vdim} \bigl( |a_2(r)| + |a_5(r)| \bigr) = 0 \,.
    \end{equation}
    From~\cite[11.4]{Fang2018}) derive 
    $ 0 \le 1 - | \tbwedge_{\vdim} \project{T}\circ \project{S} | \le 2^{2k+3} \| \project{S} -
    \project{T} \|^2 $ for $S \in \grass{\adim}{\vdim}$; thus,
    \begin{gather}
        |a_4(r)| \le r^{-1} 2^{2k+3} \sup |\zeta_\lambda'|\,
        \textint{\cball 0r \times \grass{m}{\vdim}}{}
        \| \project{S} - \project{T} \|^2
        \ud W(x,S)
        \\
        \label{eq:perp:A4estimate}
        \text{and} \quad \lim_{r\to 0} r^{-{\vdim}} |a_4(r)| =0\,.
    \end{gather}
     
    To deal with $a_3$, choose $x \in \cball 0r$ such that $\project{T}x \neq 0$, set
    $ y = \project{T}x $, and notice that
    \begin{gather}
        u \bullet x = u \bullet y \,,
        \quad
        |\perpproject{T} x|^2
        = |x - y|^2
        = |x|^2 - |y|^2
        = \bigl( |x| - |y| \bigr) \bigl( |x| + |y| \bigr) \,;
    \end{gather}
    hence, recalling $\zeta_\lambda'(t) = 0$ for $t \le \lambda $, one estimates
    \begin{multline}
        \bigl|
        u \bullet \bigl( \grad \rchi_{\lambda, r}( x ) - \grad \rchi_{\lambda, r}(y) \bigr)
        \bigr|
        = r^{-1} \bigl| \zeta_\lambda'(|x|/r) \tfrac{u \bullet  x}{|x|}
        - \zeta_\lambda'(|y|/r) \tfrac{u \bullet  y}{|y|}
        \bigr| 
        \\
        \leq   r^{-1}  |\zeta_\lambda'(|x|/r)|  
        \bigl| (u \bullet y) \bigl(\tfrac{1}{|y|} - \tfrac{1}{|x|}\bigr) \bigr| +
        r^{-1} \bigl| \zeta_\lambda'(|x|/r) - \zeta_\lambda'(|y|/r) \bigr|  
        \\
        \leq r^{-1} (\sup |\zeta_\lambda'|) \CF{\spt \zeta_\lambda'}(|x|/r)
        \tfrac{1}{|x|} \bigl| |y| - |x| \bigr| + r^{-2} \Lip \zeta_\lambda' \bigl| |x| - |y| \bigr|
        \\
        \leq \bigl( \lambda^{-1} \,\sup | \zeta_\lambda'| + \Lip | \zeta_\lambda |\bigr)\, r^{-2}\, | |x| - | y|| \\
        \le  r^{-2} \bigl(\Lip \zeta_\lambda' +  2\sup| \zeta_\lambda'|\bigr) \tfrac{|\perpproject{T} x|^2}{|x| + |y|}\,. 
    \end{multline}
    Since $|x| + |\perpproject{T}(x)| \le r/2$ implies
    $\grad \rchi_{\lambda, r}( x ) = \grad \rchi_{\lambda, r}( \perpproject{T}(x) ) = 0$ this leads
    to
    \begin{displaymath}
        | a_3(r) | \leq 2r^{-3}\, \bigl(\Lip \zeta_\lambda' +  2\sup| \zeta_\lambda'|\bigr)\,
        \textint{\cball 0r}{}|\perpproject{T} x|^2 \ud \| W \|(x)
    \end{displaymath}
    and, using~\eqref{eq:perp:height_decay_for_W}, to 
    \begin{equation}
        \label{eq:perp:A3_estimate}
        \lim_{r\to 0}r^{-{\vdim}}|a_3(r)| = 0 \,.
    \end{equation}

    In conclusion, 
    \begin{displaymath}
        \lim_{r \to 0^+} r^{-\vdim} \delta_{G} W(\rchi_{\lambda, r}\, u) = 0
        \quad \text{for every $ 0 < \lambda < 1 $ and $ u \in T \cap \sphere{m-1} $} \,.
    \end{displaymath}
    On the other hand, employing~\eqref{eq:perp:Lebesgue_points},
    \ref{mr:Lebesgue_points_and_taking_image}, \ref{rem:integral_representation_of_first_variation},
    and~\ref{def:wmcv}, one obtains
    \begin{multline}
        \lim_{r \to 0^+} r^{-\vdim}
        \bigl|
        \delta_{G} W(\rchi_{\lambda, r}\, u) - \mu \unitmeasure{\vdim} u \bullet g(0)
        \bigr|
        \le \lim_{r \to 0^+} r^{-\vdim}
        \textint{\cball{0}{r} \without \cball{0}{\lambda r}}{}
        |g| \ud \|W\|
        \\
        = \mu (1 - \lambda^{\vdim}) \unitmeasure{\vdim} |g(0)|
        \quad \text{for each $ 0 < \lambda < 1 $ and $ u \in T \cap \sphere{m-1} $} \,.
    \end{multline}
    It follows that $ g(0) \bullet u $ for every $ u \in T \cap \sphere{m} $. Finally, noting that
    $ \varphi(u) = u $, one concludes
    \begin{displaymath}
        \bigl| \tbwedge_{\vdim} \varphi \circ \project{T} \bigr| \,  h(0) \bullet u
        = \bigl| \tbwedge_{\vdim} \varphi \circ \project{T} \bigr|\,  \varphi^{*} h(0) \bullet u
        = g(0)  \bullet u  = 0
        \quad \text{for every $ u \in T \cap \sphere{m} $} \,.
        \qedhere
    \end{displaymath}
\end{proof}

\begin{remark}
    \label{rem:perp:comparison_to_Brakke}
    Brakke's proof in \cite[\S{5.8}]{Brakke1978} of the perpendicularity theorem in the Euclidean
    setting employs the Euclidean structure in an essential way. Indeed, the Pythagorean theorem
    plays a crucial role in the estimate for~$a_3$ and the estimate for~$a_0$ wouldn't go through if
    we did not ensure that~\ref{thm:mcv-orto}\eqref{eq:perp:G_quadratic} holds. Also the estimate
    for $a_1$ would not be possible if $P_G(T) \ne \project{T}$. Fortunately, perpendicularity of
    the mean curvature vector at a~point~$b$ depends only on the behaviour of the varifold in
    vicinity of~$b$ so straightening the metric at that~single point~$b$, by pushing both the
    varifold and the integrand with the linear diffeomorphism~$\psi$, solves all the problems coming
    from anisotropy.
\end{remark}

\begin{remark}
    \label{rem:perp:bad_set_estimates}
    In the isotropic case assumption~\ref{thm:mcv-orto}\eqref{eq:perp:bad_set_decay} is generic by
    means of multivalued Lipschitz approximation; cf.~\cite[\S{5.2}]{Brakke1978}
    or~\cite[Lemma~5.7]{Menne2011b}. In such approximation the part of the varifold not covered by
    the graph of a multivalued Lipschitz function (the \emph{bad set}) is controlled by the
    tilt-excess and the total variation raised to the power~$\vdim/(\vdim-1)$, which, together
    with~\ref{thm:mcv-orto}\eqref{eq:perp:height_decay} combined with the Caccioppoli-type
    inequality, yields decay faster than~$r^{\vdim+1}$. In the anisotropic case, the only available
    Lipschitz approximation theorem~\cite[\S{2.6}]{Allard1986} (which can be generalised to multiple
    values preserving the estimates) controls the bad set in a ball of radius~$r$ with the
    tilt-excess and $r$~times the total variation -- at a~generic point this is of order
    precisely~$r^{\vdim+1}$ with non-zero limit as $r \to 0^+$; hence, does not provide the decay
    required in~\ref{thm:mcv-orto}\eqref{eq:perp:bad_set_decay}.

    In the setting of~\ref{main:locality}, however, we first apply the regularity theorem
    of~\cite{KolSan2} to deduce that the unit density layer~$Q$ of~$V$ is almost everywhere locally
    associated to a~graph of a~$\cnt{1,\alpha}$ function and then
    \begin{displaymath}
         T \cap \cball{P_F(a,T)a}{r} \without X(b,T,r) = \varnothing
    \end{displaymath}
    for $V$ almost all $(b,T)$ with $b\in Q$ and $r > 0$ small enough; thus,
    condition~\ref{thm:mcv-orto}\eqref{eq:perp:bad_set_decay} is satisfied. Reliance
    on~\cite{KolSan2}, and indirectly on~\cite{Allard1986}, causes our theorem to be applicable only
    on~$Q$ rather than on the whole support of~$V$.
\end{remark}

\begin{remark}
    \label{rem:perpendicularity_in_the_smooth_case}
    If a constant density varifold $V$ inside a ball $\oball br$ is associated to the graph~$M$ of
    a~$\cnt{2}$ function, then for each $v \in \Tan(M,b)$ there exists a~tangent
    $\cnt{1}$~vectorfield $g \in \VF(M)$ supported in~$\oball br$ such that $g(b) = v$ and the flow
    of~$g$ does not change~$V$ so \mbox{$\delta_F V(g) = 0$}; hence, perpendicularity
    of~$\mcv{F}(V,b)$ follows. For~$V$ of lower regularity this approach is not applicable.
\end{remark}

\renewcommand{\adim}{{n+1}}
\renewcommand{\vdim}{n}
\renewcommand{\codim}{1}

\begin{theorem}
    \label{thm:locality}
    Suppose $\Omega \subseteq \R^{\adim}$ is open, $F : \grass{\adim}{\vdim} \to \R$ is associated
    to the norm~$\phi$ as in~\ref{mr:anisotropic_mean_curvature_vector},
    $ V \in \IVar{\vdim}(\Omega) $ such that $ \| \delta_F V \| $ is a Radon measure over $ \Omega $
    and $ M \subseteq \Omega $ is a $\vdim$-dimensional submanifold of class~$\cnt{2}$ such that
    $ \HM^{\vdim}(M) < \infty $.

    Let $A$ be the set of $a \in  M$ such that
    \begin{gather}
        \label{eq:locality:perpendicularity}
        \mcv{F}(V,a) \perp \Tan^{\vdim}(\|V\|,a) \in \grass{\adim}{\vdim} \,
        \\ \text{and} \quad
        \label{eq:locality:tilt-decay}
        \limsup_{r \to 0^{+}} r^{-\vdim-2}
        \textint{\cball ar \times \grass{\adim}{\vdim}}{}
        \|\project{S} - \project{\Tan^{\vdim}(\|V\|,a)}\|^2
        \ud V(x, S) < \infty
        \,.
    \end{gather}
    Then
    \begin{displaymath}
        \mcv{F}(V,a) = \mcv{F}(M,a)
        \quad \text{for $\|V\|$ almost all $a \in A$} \,.
    \end{displaymath}
\end{theorem}

\begin{proof}
    We set $ \tau(a) = \project{\Tan^n(\| V \|, a)} $ whenever
    $\Tan^{\vdim}(\|V\|,a) \in \grass{\adim}{\vdim} $.  In particular, by
    \ref{rem:integral_representation_of_first_variation}, \cite[3.5(1)]{Allard1972} and \cite[2.9.9,
    2.9.10]{Federer1969} we see that for $ \| V \| $ a.e.\ $ a \in \Omega $ there exists
    $ m \in \natp $ such that
    \begin{gather}
        \density^{\vdim}(\|V\|,a) = m, \quad \density^{\vdim}(\|\delta_F V\|_{\mathrm{sing}},a)=0 \,,
        \\
        \VarTan(V,a) = \bigl\{ m\, \var{n}(\Tan^n(\| V \|,a)) \bigr\} 
    \end{gather}
    and $ a $ is a $\|V\|$~Lebesgue point of $(F \circ \tau)\,\mcv{F}(V,\cdot) $. Moreover, by
    \cite[2.10.19(4), 3.2.16]{Federer1969} we have that
    \begin{displaymath}
        \density^{\vdim}(\|V\| \restrict \R^{\adim} \without M,a)  = 0
        \quad \textrm{and} \quad \Tan^n(\| V \|,a) = \Tan(M,a)
    \end{displaymath}
    for $ \| V \| $ a.e.\ $ a \in M $. For $d \in \natp$ we define
    \begin{displaymath}
        A_d = A \cap \left\{
            a :
            \begin{gathered}
                \mcv{F}(V,a) \perp \Tan(M,a) = \Tan^{\vdim}(\|V\|,a) ,\,
                \\
                \text{$a$ is a $(\|V\|,C)$~Lebesgue point of $(F \circ \tau)\,\mcv{F}(V,\cdot)$} ,\,
                \\
                \density^{\vdim}(\|V\|,a) = d = \density^{\vdim}(\|V\| \restrict M,a) \,,
                \density^{\vdim}(\|\delta_F V\|_{\mathrm{sing}},a) = 0,\\
                \VarTan(V,a) = \{d\, \var{n}(\Tan^n(\| V \|,a))\}
            \end{gathered}
        \right\} \,
    \end{displaymath}
    and, recalling hypothesis \eqref{eq:locality:perpendicularity}, we see that
    \begin{displaymath}
        \| V \| \bigl(A \without {\textstyle \bigcup_{d \in \natp}}A_d \bigr) =0 \,.
    \end{displaymath}

    We choose $ d \in \natp $ and $ a \in A_d $ so that
    \begin{displaymath}
        \density^n(\HM^{\vdim} \restrict (M \without A_d), a) = \density^n(\| V \| \restrict \R^{\adim}
        \without A_d, a) =0
    \end{displaymath}
    and we choose
    \begin{displaymath}
        T = \Tan(M,a)\,, \quad
        \nu \in T^{\perp} \cap \sphere{\vdim}  \,,
        \quad
        \eta = \grad \phi(\nu) 
    \end{displaymath}
    and $\zeta : \R \to \R$ a smooth such that $\zeta(t) = 0$ for $1 \le t < \infty$. Define
    \begin{displaymath}
        g_{r}(z) = \zeta(|z-a|/r) \eta
        \quad \text{for $0 < r < \infty$} \,.
    \end{displaymath}
   Then we observe that
    \begin{multline}
        \label{eq:locality_mcv_expressions}
        - d \phi(\nu) \bigl( \mcv{F}(V,a) \bullet \eta \bigr)
        \textint{T \cap \cball 01}{} \zeta(|z|) \ud \HM^{\vdim}(z)
        = \lim_{r \to 0^{+}} r^{-\vdim} \delta_F V(g_r)
        \\
        = \lim_{r \to 0^{+}} r^{-\vdim} \textint{}{} B_{F}(S) \bullet \uD g_r(z) \ud V(z,S)
        \,.
    \end{multline}
    Set $W = d\,\var{\vdim}(M)$ and observe that since $M$ is of class~$\cnt{2}$ properties
    \eqref{eq:locality:tilt-decay} and~\eqref{eq:locality_mcv_expressions} hold also for~$W$
    in~place of~$V$. Recalling that $\mcv{F}(V,a), \mcv{F}(W,a) \in \im T^{\perp}$,
    $\dim \im T^{\perp} = 1$, and $\eta \bullet \nu = \phi(\nu) \ne 0$, it~suffices to prove that
    \begin{displaymath}
        \lim_{r \to 0^{+}}
        r^{-\vdim} \textint{}{} B_{F}(S) \bullet \uD g_r(z) \ud V(z,S)
        - r^{-\vdim} \textint{}{} B_{F}(S) \bullet \uD g_r(z) \ud W(z,S) = 0 \,.
    \end{displaymath}
    From~\ref{mr:anisotropic_mean_curvature_vector} we see that $\im \uD g_r(z) \perp \im B_{F}(T)$ and
    $B_{F}(T) \bullet \uD g_r(z) = 0$ for $z \in \Omega$. We observe that
    \begin{multline}
        \textint{}{} B_{F}(S) \bullet \uD g_r(z) \ud V(z,S)
        - \textint{}{} B_{F}(S) \bullet \uD g_r(z) \ud W(z,S)
        \\
        = \textint{(\cball ar \without A_d) \times \grass{\adim}{\vdim}}{}
        (B_{F}(S) - B_{F}(T)) \bullet \uD g_r(z)
        \ud V(z,S)
        \\
        - \textint{(\cball ar \without A_d) \times \grass{\adim}{\vdim}}{}
        (B_{F}(S) - B_{F}(T)) \bullet \uD g_r(z)
        \ud W(z,S)
        \,.
    \end{multline}
    Define
    \begin{multline}
        R_{U}(r) = \textint{(\cball ar \without A_d) \times \grass{\adim}{\vdim}}{}
        (B_{F}(S) - B_{F}(z,T)) \bullet \uD g_r(z)
        \ud U(z,S)
        \\
        \text{for $0 < r < \infty$ and $U \in \{ V ,\, W \}$} \,.
    \end{multline}
    We shall show that $\lim_{r \to 0^{+}} r^{-\vdim} R_U(r) = 0$ for $U \in \{ V ,\, W \}$. Since
    $\phi$ is of class~$\cnt{2}$ we see that $B_{F}$ is Lipschitzian. Set
    $L = \Lip B_{F} \Lip \zeta$ and note that $|\eta| \le c_{1}(\phi)$. For brevity define
    $E(r) = (\cball ar \without A_d) \times \grass{\adim}{\vdim}$ for $0 < r < \infty$. We~estimate
    \begin{multline}
        (L c_1(\phi))^{-1} r^{-\vdim} R_{U}(r)
        \le r^{-\vdim-1} \textint{E(r)}{}
        \| \project{S} - \project{T} \| \ud U(z,S)
        \\
        \le \bigl( r^{-\vdim} \|U\| ( \cball ar \without A_d ) \bigr)^{1/2}
        \Bigl( r^{-\vdim-2} \textint{E(r)}{}
        \| \project{S} - \project{T} \|^2 \ud U(z,S)
        \Bigr)^{1/2} \,.
    \end{multline}
    The first term clearly converges to~$0$ and the second stays bounded as $r \to 0^{+}$; hence,
    the theorem is proven.
\end{proof}

\section{Proof of the main theorem}
\label{sec:regularity}

In this section we combine the main result from our previous paper~\cite{KolSan2}
with~\ref{thm:locality} to prove~\ref{main:locality}.

\begin{proof}
    Define $ A = \spt \| V \| $ and
    \begin{displaymath}
        R = A \cap \bigl\{
        x 
        : \exists \, r > 0 \ 
        \exists \, \nu \in \sphere{\vdim} \quad
        \oball{a+r\nu}{r} \cap A = \varnothing = \oball{a-r\nu}{r} \cap A
        \bigr\} \,.
    \end{displaymath}
    From~\cite[Theorem~1.1]{KolSan2} it follows that $\HM^{\vdim}(\spt \|V\| \without R) = 0$ and
    applying~\cite[Theorem~1.2]{KolSan2} one gets also that for $\HM^{\vdim}$ almost all $a \in Q$
    there exists $0 < r_a < \infty$ such that
    \begin{equation}
        \label{eq:regularity}
        \oball{a}{r_a} \cap \spt \|V\| = \oball{a}{r_a} \cap Q
        \quad \text{is a $\cnt{1,\alpha}$~hypersurface} \,.
    \end{equation}
    Clearly for each $a \in R$ there holds
    \begin{equation}
        \label{eq:mreg:height_decay}
        \lim_{r \to 0^{+}} r^{-\vdim-4} \textint{\cball ar}{} | \perpproject{T} x |^2 \ud \|V\|(x) < \infty
    \end{equation}
    so employing~\ref{cor:tilt-decay} one gets also
    \begin{equation}
        \label{eq:mreg:tilt_decay}
        \limsup_{r \to 0^{+}} r^{-\vdim-2} \textint{\cball ar \times \pgrass{\adim}{\vdim}}{}
        \| \project{S} - \project{T} \|^2 \ud V(z,S) < \infty \,.
    \end{equation}
    Using~\eqref{eq:regularity} one sees that for $\HM^{\vdim}$ almost all $a \in Q \cap R$ there is
    $0 < r < r_a$ so small that condition~\ref{thm:mcv-orto}\eqref{eq:perp:bad_set_decay} is
    satisfied (cf.~\ref{rem:perp:bad_set_estimates}); hence, since~\eqref{eq:mreg:height_decay}
    and~\eqref{eq:mreg:tilt_decay} ensure validity of the remaining hypotheses, \ref{thm:mcv-orto}
    yields that
    \begin{equation}
        \label{eq:mreg:perpendicularity}
        \mcv{F}(V,a) \perp \Tan^{\vdim}(\|V\|,a)
        \quad \text{for $\|V\|$~almost all $a \in Q \cap R$}\,.
    \end{equation}
    Let $M \subseteq \Omega$ be a properly embedded $\vdim$-dimensional submanifold of
    class~$\cnt{2}$. Combining~\eqref{eq:mreg:perpendicularity} with~\eqref{eq:mreg:tilt_decay} and
    applying~\ref{thm:locality} we see that for $\|V\|$~almost all $a \in Q \cap M \cap R$ there
    holds $\mcv{F}(V,a) = \mcv{F}(M,a)$.
\end{proof}

\subsection*{Acknowledgements}
The~research of Sławomir Kolasiński was financed by the \href{https://ncn.gov.pl/}{National Science
  Centre Poland} grant number 2022/46/E/ST1/00328.
\\
The research of Mario Santilli is partially supported by INDAM-GNSAGA and PRIN project 20225J97H5.

{\small
 \addcontentsline{toc}{section}{\numberline{}References}
 \bibliography{santilli1}{}
 \bibliographystyle{myalphaurl}
}

\bigskip

{\small \noindent
  Sławomir Kolasiński \\
  Uniwersytet Warszawski, Instytut Matematyki \\
  ul. Banacha 2, 02-097 Warszawa, Poland \\
  \texttt{s.kolasinski@mimuw.edu.pl}
}

\bigskip

{\small \noindent
  Mario Santilli \\
  Department of Information Engineering, Computer Science and Mathematics,\\
  Università degli Studi dell'Aquila\\
  via Vetoio 1, 67100 L’Aquila, Italy\\
  \texttt{mario.santilli@univaq.it}
}

\end{document}